\documentclass[12pt]{article}
\usepackage{amsmath, amssymb, amsthm}
\usepackage{graphicx, subfigure, caption}
\usepackage{booktabs}
\usepackage{algorithm}
\usepackage{geometry}
\usepackage{hyperref}
\usepackage[numbers,sort&compress]{natbib}

\numberwithin{equation}{section}

\newtheorem{theorem}{Theorem}[section]
\newtheorem{lemma}{Lemma}[section]
\newtheorem{proposition}{Proposition}[section]
\newtheorem{example}{Example}[section]

\newtheorem{corollary}{Corollary}[section]
\newtheorem{definition}{Definition}[section]
\newtheorem{remark}{Remark}[section]

\title{Some properties of the solution of the vertical tensor complementarity problem\thanks{This research was supported by National Natural Science Foundations of
China (No. 11961082).}}
\author{Li-Ming Li\thanks{llm13508044340@163.com}, \quad
        Shi-Liang Wu\thanks{Corresponding author: slwuynnu@126.com}\\
        {\small{\it  School of Mathematics, Yunnan Normal University, }}\\
        {\small{\it  Kunming, Yunnan, 650500, PR China}}}
\date{}

\begin{document}
\maketitle

\begin{abstract}
In this paper, we mainly focus on the existence and uniqueness of
the vertical tensor complementarity problem. Firstly, combining the
generalized-order linear complementarity problem with the tensor
complementarity problem, the vertical tensor complementarity problem
is introduced. Secondly, we define some sets of special tensors, and
illustrate the inclusion relationships. Finally, we show that the
solution set of the vertical tensor complementarity problem is
bounded under certain conditions, and some sufficient conditions for
the existence and uniqueness of the solution of the vertical tensor
complementarity problem are obtained from  the view of the degree
theory and the equal form of the minimum function.



\end{abstract}

\section{Introduction}
Over the past few decades, many mathematical researchers have conducted
extensive research on the nonlinear complementarity problem. Today,
the nonlinear complementarity problem has become a classical problem,
and has important applications in operations research and applied science.
For more detailed information, one can see
\cite{cottle1992linear,ferris1997engineering,harker1990finite,han2006nonlinear,facchinei2003finite}
and references therein.

In recent years, with the rise of tensor research, a special case of
the nonlinear complementarity problem: the tensor complementarity
problem, has also received widespread attention and research. The
tensor complementarity problem was first proposed in
\cite{song2014properties} by Song and Qi, i.e., find a vector $x\in
R^{n}$ or show that there is no vector such that
\begin{equation*}
    x\ge\theta,
    q+\mathcal{A} x^{m-1} \ge\theta,
    x^{T}(\mathcal{A} x^{m-1}+q) =0,
\end{equation*}
where $\theta$ denotes the zero vector, $\mathcal{A}$
and $q\in R^{n}$ are a given $m$
order $n$ dimensional tensor and a vector,
respectively, and $\mathcal{A} x^{m-1}$ in \cite{qi2005eigenvalues} is
defined as
\begin{equation*}
\left(\mathcal{A} x^{m-1}\right)_{i}=\sum_{i_{2}, \cdots, i_{n}}^{n}
a_{ii_{2}\cdots i_{n}}x_{i_{2}}\cdots x_{i_{n}},i\in[n],
\end{equation*}
which is denoted by TCP$\left(\mathcal{A}, q\right)$. They studied
the existence of the solution of the TCP$\left(\mathcal{A},
q\right)$ through the structure of the tensor, and obtained a
sufficient and necessary condition for the non-negative symmetric
tensor $\mathcal{A}$. Bai et al. \cite{bai2016global} studied the
global uniqueness and the solvability of the TCP$\left(\mathcal{A},
q\right)$, and obtained two results: $\left(1\right)$ if
$\mathcal{A}$ is a $P$-tensor, then the
TCP$\left(\mathcal{A},q\right)$ may have more than one solution for
some $q \in R^{n}$; $\left(2\right)$ if $\mathcal{A}$ is a strong
$P$-tensor, then the TCP$\left(\mathcal{A}, q\right)$ has
GUS-property for all $q \in R^{n}$. The readers can see
\cite{che2016positive,gowda2015z,huang2017formulating,luo2017sparsest,song2015properties}
to know more about the existence of the solution of the
TCP$\left(\mathcal{A}, q\right)$. In addition, many researchers are
committed to studying the numerical algorithm for solving the
TCP$\left(\mathcal{A}, q\right)$, see
\cite{huang2017formulating,song2016tensor,du2018tensor,du2019mixed,han2019continuation,liu2018tensor,wang2022randomized,dai2022gus}.

Goeleven \cite{goeleven1996uniqueness} introduced the
generalized-order linear complementarity problem  and studied the
existence and uniqueness of its solution. The generalized-order
linear complementarity problem has important applications in
stochastic impulse problems, mixed lubrication problems and singular
control problems in bounded intervals, see
\cite{mosco1976implicit,oh1986formulation,min1987singular,sun1989monotonicity}.

In this paper, combining the generalized-order linear
complementarity problem and the TCP$\left(\mathcal{A}, q\right)$, it
is natural to introduce the vertical tensor complementarity problem,
i.e., find a vector $x\in R^{n}$ or show that there is no vector
such that
\begin{equation*}
    q_{1}+\mathcal{A}_{1} x^{m-1} \ge\theta,
    q_{2}+\mathcal{A}_{2} x^{m-1} \ge\theta,
    (q_{1}+\mathcal{A}_{1} x^{m-1})^{T}(q_{2}+\mathcal{A}_{2} x^{m-1}) =0,
\end{equation*}
where $\mathcal{A}_{1}$, $\mathcal{A}_{2} $ are given $m$ order $n$
dimensional tensors, and $q_{1}$, $q_{2}\in R^{n}$ are given
vectors. We denote it by VTCP$(\tilde{\mathcal{A}} , \tilde{q })$
with $\tilde{\mathcal{A} } = \left \{ \mathcal{A}_{1},
\mathcal{A}_{2} \right \}$ and $\tilde{q}= \left \{ {q}_{1}, {q}_{2}
\right \}$. We consider the VTCP$(\tilde{\mathcal{A}} , \tilde{q })$
for two reasons: (1) it is equal to the tensor absolute value
equation, see \cite{du2018tensor}; (2) it is the generalization form
of the TCP$\left(\mathcal{A}, q\right)$, too. To our knowledge, so
far, for the unique solution of the VTCP$(\tilde{\mathcal{A}} ,
\tilde{q })$, the necessary and sufficient condition is
\emph{scarce}, which is our main motivation. Hence, based on this,
our aim is to address this problem, and obtain some useful results
on the sufficient condition for the unique solution of the
VTCP$(\tilde{\mathcal{A}} , \tilde{q })$.

It is a brief description for the rest of paper. In Section 2, we go
over some preliminary knowledge. In Section 3, we introduce some
tensor sets with special structures and demonstrate their
properties. In Section 4, we show that the solution set of the
VTCP$(\tilde{\mathcal{A}} , \tilde{q })$ is bounded under certain
conditions, and some sufficient conditions for the existence and
uniqueness of the solution of the VTCP$(\tilde{\mathcal{A}} ,
\tilde{q })$ are obtained on the base of the degree theory and the
equal form of minimum function. In Section 5, some conclusions are
given to end this paper.
\section{Preliminaries}

In this section, we review some definitions and notations.

We use $\mathcal{A}$, $\mathcal{B}$, $\mathcal{C}$ to represent the
tensors, $A$, $B$, $C$ to represent the matrices, $x$, $y$, $z$ to
represent the vectors. We use $x_{i}$ to represent the $i$th element
of the vector $x$, $I$ to represent the unit tensor (including the unit
vector and the identity matrix), $\theta$ to represent the zero tensor (including the zero vector and the zero matrix).
$x^{\left[m\right]}$ is a vector that
$\left(x^{\left[m\right]}\right)_{i}=\left(x_{i}\right)^m$, $i=1,
\cdots, n$. Given a vector $x \in R^{n}$,
\begin{equation*}
\left(x\right)_{+}:=(\max\left\{x_{1},0\right\},
\max\left\{x_{2},0\right\}, \cdots, \max\left\{x_{n},0\right\})
\end{equation*}
and
\begin{equation*}
\left(x\right)_{-}:=\left(\text{max}\left\{-x_{1},0\right\},
\text{max}\left\{-x_{2},0\right\}, \cdots,
\text{max}\left\{-x_{n},0\right\}\right).
\end{equation*}
Given $x, y\in R^{n}$, $x\wedge y\left({\rm or}~{\rm min}\left\{x,
y\right\}\right)$ represents the vector, whose $i$th element is equal
to \textrm{min}$\left(x_{i}, y_{i}\right)$; $x \vee  y\left({\rm
or}~{\rm max}\left\{x, y\right\}\right)$ represents the vector, whose
$i$th element is equal to \textrm{max}$\left(x_{i}, y_{i}\right)$.
$x>y$, $x \ge y$, $x<y$ and $x \le y$ mean $x_{i}>y_{i}$, $x_{i} \ge
y_{i}$, $x_{i}<y_{i}$ and $x_{i} \le y_{i}$, respectively, with
$i=1, \cdots, n$. $R_{+}^{n}:=\left \{ x \in R^{n}:x_{i}\ge \theta,
i=1, \cdots, n  \right \}$ and $R_{++}^{n}:=\left \{ x \in
R^{n}:x_{i}> \theta, i=1, \cdots, n  \right \}$.
$\left[n\right]:=\left\{1, \cdots , n\right\}$.

A tensor $\mathcal{A}=\left(a_{i_{1} i_{2}\cdots i_{m}}\right)$ is a
multidimensional array of elements represented as $a_{i_{1}
i_{2}\cdots i_{m}}\in F$, where $i_{j}=1, 2, \cdots, n_{j}$, $j=1,
2, \cdots, m$ and $F$ represents a field. In this paper, we consider
the problem in the field of real numbers, i.e., $F=R$. $m$ is called
the order of $\mathcal{A}$, and $\left(n_{1}, \cdots, n_{m}\right)$
is called the  dimension of $\mathcal{A}$. Obviously, when $m$= 2,
the tensor $\mathcal{A}$ is a matrix. $\mathcal{A}$ is called an $m$
order $n$ dimensional tensor, if $n_{1}=n_{2}=\cdots=n_{m}$.
$R^{\left[m, n\right]}$ represents the set of all $m$ order $n$
dimensional tensors. $a_{i,\cdots,i}$ is called the principal
diagonal element of $\mathcal{A}\in R^{\left[m, n\right]}$,
$i=1,\cdots,n$. $\mathcal{A}$ is called a symmetric tensor if
$a_{i_{1} i_{2}\cdots i_{m}}$ is invariant  no matter how $i_{1},
i_{2}, \cdots, i_{m}$ are arranged. Denote the set of all $m$ order
$n$ dimensional real symmetric tensors as $S^{\left[m, n\right]}$.

The generalized product of tensors \cite{shao2013general} is
required. Let $\mathcal{A}$ and $\mathcal{B}$ are $n$ dimensional
tensors with order $m\left(m \ge 2\right)$ and $k\left(k \ge
1\right)$, respectively. Define the product $\mathcal{A} \cdot
\mathcal{B}$  to be the following tensor $\mathcal{C}$ of order
$\left(m-1\right)\left(k-1\right)$ + 1 and dimension $n$:
\begin{equation*}
\mathcal{C}_{i\alpha _{1}\cdots \alpha _{m-1}}=  \sum_{i_{2}, \cdots
, i_{n}=1}^{n} a_{ii_{2}\cdots   i_{m}}b_{i_{2}\alpha _{1}}\cdots
b_{i_{m}\alpha _{m-1}} \left ( i \in \left [ n \right ], \alpha
_{1},\cdots, \alpha _{m-1}\in \left [ n \right ]^{k-1}     \right ).
\end{equation*}
By Theorem 1.1 and Example 1.1 in \cite{shao2013general}, the
generalized tensor product satisfies the associative law, and
$\mathcal{A} x^{m-1}$ can be written as $\mathcal{A}\cdot x$. By
Proposition 1.1 in \cite{shao2013general}, the generalized tensor
product satisfies
$\left(\mathcal{A}_{1}+\mathcal{A}_{2}\right)\cdot\mathcal{B}=\mathcal{A}_{1}\cdot\mathcal{B}+\mathcal{A}_{2}\cdot\mathcal{B}$,
where $\mathcal{A}_{1}$ and $\mathcal{A}_{2}$ have the same order, and
also satisfies
$A\cdot\left(\mathcal{B}_{1}+\mathcal{B}_{2}\right)=A\cdot\mathcal{B}_{1}+A\cdot\mathcal{B}_{2}$,
where $A$ is a matrix.

Based on the generalized product of tensors, the left and right
inverses of tensors \cite{bu2014inverse} are defined. Let
$\mathcal{A} \in C_{t}^{\left ( n, n \right ) }  $. If there is a
tensor $\mathcal{B} \in C_{k}^{\left ( n, n \right ) }$ such that
$\mathcal{A}\cdot\mathcal{B}=\mathcal{I} $, where $C_{k}^{\left ( m,
n \right ) }$ represents the set of all complex tensors with order
$k$ and dimension $\left(m, n, \cdots, n\right)$, then it is said
that $\mathcal{A} $ is  left inverse of $\mathcal{B}$, and is
denoted as $\mathcal{B}^{L_{t}}$; $\mathcal{B} $ is  right inverse
of $\mathcal{A}$, and is denoted as $\mathcal{A}^{R_{k}}$.

Inspired by the definition of $\mathcal{A} x^{m-1}$, we define the
following $\mathcal{A}\mathcal{B}$ as the product of $\mathcal{A}\in
R^{\left[m, n\right]}$ and $\mathcal{B}\in R^{\left[m-1, n\right]}$:
\begin{equation*}
\left(\mathcal{A}\mathcal{B}\right)_{i}=\sum_{i_{2}, \cdots  ,
i_{n}}^{n}a_{ii_{2}\cdots   i_{n}}b_{i_{2}\cdots   i_{n}}\left(i
\in[n]\right).
\end{equation*}

Now, we review the $R$-tensor \cite{song2014properties} and the
$Z$-tensor \cite{ding2003mtensor}. $\tilde{\mathcal{A} }$ is called
a $R$-tensor if and only if  the following system is inconsistent:
\begin{equation*}
    \begin{cases}
        \theta\ne x \ge\theta, t\ge 0, \\\left ( \mathcal{A} x^{m-1} \right )_{i}+t=0, \text{ if }  x_{i}>0,
        \\\left ( \mathcal{A} x^{m-1} \right )_{j}+t\ge 0, \text{ if }  x_{j}=0.
       \end{cases}
\end{equation*}
$\mathcal{A}$ is called a $\mathcal{Z}$-tensor if and only if all
its off-diagonal entries are non-positive.

\section{Sets of tensors with special structures}
In this section, we introduce some sets of tensors with special structures.

\begin{definition}\label{def:1}
    Let $\tilde{\mathcal{A} }=\left \{ \mathcal{A}_{1}, \mathcal{A}_{2}   \right \}$ with $\mathcal{A}_{1}, \mathcal{A}_{2}  \in R^{\left [ m, n  \right ] }$. We say that $\tilde{\mathcal{A} }$ is of
    \begin{enumerate}
        \item type $VR_{0}$ if $\mathcal{A}_{1}\cdot x \wedge \mathcal{A}_{2}\cdot x=\theta \Rightarrow x=\theta$;
        \item type $VE$ if $\mathcal{A}_{1}\cdot x \wedge \mathcal{A}_{2}\cdot x\le\theta \Rightarrow x=\theta$;
        \item type $VP$ if $\mathcal{A}_{1}\cdot x \wedge \mathcal{A}_{2}\cdot x\le\theta\le \mathcal{A}_{1}\cdot x \vee  \mathcal{A}_{2}\cdot x \Rightarrow x=\theta$;
        \item type $VP$-$\uppercase\expandafter{\romannumeral1}$ if $\mathcal{A}_{1}\cdot x \wedge \mathcal{A}_{2}\cdot x\le\theta\le \mathcal{A}_{1}\cdot x \vee  \mathcal{A}_{2}\cdot x $ have no solution in $\left ( R_{+}^{n}/\left\{\theta \right\} \right )     \cup  \left ( R_{-}^{n}/\left\{\theta \right\} \right )$;
        \item type $VP$-$\uppercase\expandafter{\romannumeral2}$ if there is no tensor $\mathcal{Z} \in  S^{\left [ m-1, n  \right ] }$, whose principal diagonal elements are not all zeros such that $\mathcal{A}_{1} \mathcal{Z} \wedge \mathcal{A}_{2}  \mathcal{Z}\le\theta\le \mathcal{A}_{1} \mathcal{Z} \vee  \mathcal{A}_{2} \mathcal{Z} $;
        \item type semi-positive if there exists a vector $ x>\theta$ such that $\mathcal{A}_{1}\cdot x >\theta$ and $\mathcal{A}_{2} \cdot x >\theta$;
        \item type $VQ$ if the VTCP$(\tilde{\mathcal{A}} , \tilde{q})$ has a solution for all $q_{1}, q_{2}\in R^{n}$.
    \end{enumerate}
\end{definition}

\begin{remark} For Definition \ref{def:1}, the condition of item 1 is equivalent to that the following system is
inconsistent:
    \begin{equation} \label{eq:VR}
    \begin{cases}
        x\ne \theta , \\\mathcal{A}_{2}\cdot x\ge\theta, \\\left (  \mathcal{A}_{1}\cdot x \right ) _{i}=\theta  , \text{ if }  \left (  \mathcal{A}_{2}\cdot x \right ) _{i}>\theta,
        \\\left (  \mathcal{A}_{1}\cdot x \right ) _{j}\ge \theta , \text{ if }  \left (  \mathcal{A}_{2}\cdot x \right ) _{j}=\theta;

       \end{cases}
    \end{equation}
the condition of item 2 is equivalent to that $\forall x\ne \theta
$, $\exists i\in \left[n\right]$ satisfies $\left (
\mathcal{A}_{1}\cdot x \right )_{i} >\theta$, $\left (
\mathcal{A}_{2}\cdot x \right )_{i} > \theta $; the condition of
item 3 is equivalent to that $\forall x\ne \theta $, $\exists i\in
\left[n\right]$ satisfies $\left ( \mathcal{A}_{1}\cdot x \right
)_{i} \left ( \mathcal{A}_{2}\cdot x \right )_{i} >\theta $; the
condition of item 4 is equivalent to that $\forall x \in \left (
R_{+}^{n}/\left\{\theta \right\} \right )     \cup  \left (
R_{-}^{n}/\left\{\theta \right\} \right ) $, $\exists i\in
\left[n\right]$ satisfies $\left ( \mathcal{A}_{1}\cdot x \right
)_{i} \left ( \mathcal{A}_{2}\cdot x \right )_{i} >\theta $; the
condition of item 5 is equivalent to that $\forall \mathcal{Z} \in
S^{\left [ m-1, n  \right ] } $ with its principal diagonal elements
being not all zeros, $\exists i\in \left[n\right]$ satisfies $\left
( \mathcal{A}_{1} \mathcal{Z} \right )_{i} \left ( \mathcal{A}_{2}
\mathcal{Z} \right )_{i} >\theta $.

\end{remark}

\begin{proposition} \label{pro:1}
type $VE$ $\subseteq$  type  $VP$ $\subseteq$ type  $VR_{0}$; type
$VP$-$\uppercase\expandafter{\romannumeral2}$ $\subseteq$  type $VP$
$\subseteq$ type
$VP$-$\uppercase\expandafter{\romannumeral1}$.
\end{proposition}
\begin{proof}
Assume that $\tilde{\mathcal{A} }$ is of type $VE$, but not type
$VP$. Then there is a vector $x \ne \theta$ such that
    \begin{equation}\label{eq:1.0}
    \mathcal{A}_{1}\cdot x \wedge \mathcal{A}_{2}\cdot x\le\theta\le \mathcal{A}_{1}\cdot x \vee  \mathcal{A}_{2}\cdot x.
    \end{equation}
By the left inequality in \eqref{eq:1.0}, $\tilde{\mathcal{A} }$ is
not type $VE$. This contradicts the assumption.

Assume that $\tilde{\mathcal{A} }$ is of type  $VP$, but not type
$VR_{0}$. Then there is a vector $x \ne \theta$ such that
$\mathcal{A}_{1}\cdot x \wedge \mathcal{A}_{2}\cdot x=\theta. $
Furthermore,
\begin{equation*}
        \mathcal{A}_{1}\cdot x \wedge \mathcal{A}_{2}\cdot x=\theta\le \mathcal{A}_{1}\cdot x \vee  \mathcal{A}_{2}\cdot x,
\end{equation*}
which implies that $\tilde{\mathcal{A} }$ is not type $VP$. This
contradicts the assumption.

Thus, based on the above discussion, we can draw a conclusion that
type $VE$ $\subseteq$ type $VP$ $\subseteq$ type $VR_{0}$.

Similarly, type $VP$-$\uppercase\expandafter{\romannumeral2}$
$\subseteq$  type  $VP$ $\subseteq$ type
$VP$-$\uppercase\expandafter{\romannumeral1}$ clearly holds.
\end{proof}

Examples 3.1-3.3 are given to illustrate three aspects:
$\left(1\right)$ the existence of type $VE$; $\left(2\right)$ type
$VP$ is not necessarily type $VE$; $\left(3\right)$ type  $VR_{0}$
is not necessarily type $VP$.

\begin{example} \label{ex:3.1}
    Let
    $\tilde{\mathcal{A} }$ = $\left \{ \mathcal{A}_{1}, \mathcal{A}_{2}   \right \}$, where
    \begin{gather*}
        \mathcal{A}_{1}\left ( 1, :, : \right )  =\begin{pmatrix}
            -1  & 3\\0
              &0
            \end{pmatrix},~\mathcal{A}_{1}\left ( 2, :, : \right )=\begin{pmatrix}
             1 & 0\\-3
              &1
            \end{pmatrix},
        \\\mathcal{A}_{2}\left ( 1, :, : \right )  =\begin{pmatrix}
                0  & 2\\1
                  &0
                \end{pmatrix},~\mathcal{A}_{2}\left ( 2, :, : \right )=\begin{pmatrix}
                 2 & -1\\-2
                  &1
            \end{pmatrix}.
        \end{gather*}
        By the calculation, we obtain
        \begin{equation*}
            \mathcal{A}_{1}x^{2} = \begin{pmatrix}
                -x_{1}^{2} +3x_{1}x_{2}    \\x_{1}^{2} -3x_{1}x_{2}+x_{2}^{2}

                \end{pmatrix},~
            \mathcal{A}_{2}x^{2} = \begin{pmatrix}
                    3x_{1}x_{2}    \\2x_{1}^{2} -3x_{1}x_{2}+x_{2}^{2}

                \end{pmatrix}.
            \end{equation*}
Suppose $x\ne\theta$, we discuss three cases: $-x_{1}^{2}
+3x_{1}x_{2}>0$, $-x_{1}^{2} +3x_{1}x_{2}=0$ and $-x_{1}^{2}
+3x_{1}x_{2}<0$.
\begin{enumerate}
    \item [(1)] if $-x_{1}^{2} +3x_{1}x_{2}>0$, then $3x_{1}x_{2}>0; $
    \item [(2)] if $-x_{1}^{2} +3x_{1}x_{2}=0$, then $x_{2}\ne 0, x_{1}^{2} -3x_{1}x_{2}+x_{2}^{2}>0$ and $2x_{1}^{2} -3x_{1}x_{2}+x_{2}^{2}>0; $
    \item [(3)] if $-x_{1}^{2} +3x_{1}x_{2}<0$, then $x_{1}^{2} -3x_{1}x_{2}+x_{2}^{2}>0$ and $2x_{1}^{2} -3x_{1}x_{2}+x_{2}^{2}>0. $
\end{enumerate}
From the above analysis, $\tilde{\mathcal{A} }$ is of type $VE$.
\end{example}

\begin{example} \label{ex:3.2}
    Let
    $\tilde{\mathcal{A} } = \left \{ \mathcal{A}_{1}, \mathcal{A}_{2}   \right \}$, where
    \begin{gather*}
        \mathcal{A}_{1}\left ( 1, 1, :, : \right )  =\begin{pmatrix}
            1  & 2\\-2
              &0
            \end{pmatrix},~\mathcal{A}_{1}\left (1, 2, :, : \right )=\begin{pmatrix}
             0 & 0\\0
              &0
            \end{pmatrix},
        \\\mathcal{A}_{1}\left ( 2, 1, :, : \right )  =\begin{pmatrix}
            0  & -2\\1
              &0
            \end{pmatrix},~\mathcal{A}_{1}\left (2, 2, :, : \right )=\begin{pmatrix}
             2 & 0\\0
              &1
            \end{pmatrix},
        \\\mathcal{A}_{2}\left ( 1, 1, :, : \right )  =\begin{pmatrix}
            1  & 0\\0
              &1
            \end{pmatrix},~\mathcal{A}_{2}\left (1, 2, :, : \right )=\begin{pmatrix}
             0 & 1\\2
              &0
            \end{pmatrix},
        \\\mathcal{A}_{2}\left ( 2, 1, :, : \right )  =\begin{pmatrix}
            0  & 1\\1
              &0
            \end{pmatrix},~\mathcal{A}_{2}\left (2, 2, :, : \right )=\begin{pmatrix}
             0 & 0\\0
              &1
            \end{pmatrix}.
    \end{gather*}
    By the calculation, we obtain
   \begin{gather*}
        \mathcal{A}_{1}x^{3} = \begin{pmatrix}
            x_{1}^{3}   \\x_{1}^{2}x_{2}+x_{2}^{3}
            \end{pmatrix},~
        \mathcal{A}_{2}x^{3} = \begin{pmatrix}
            x_{1}^{3} +4x_{1}x_{2}^{2} \\x_{2}^{3} +2x_{1}^{2}x_{2}
            \end{pmatrix}.
    \end{gather*}
Suppose $x\ne\theta$, we discuss two cases: $x_{1}\ne 0$ and
$x_{1}=0$.
\begin{enumerate}
    \item [(1)] if $x_{1}\ne 0$, then $\left ( \mathcal{A}_{1}\cdot x \right )_{1} \left ( \mathcal{A}_{2}\cdot x \right )_{1}=x_{1}^{6} +4x_{1}^{4}x_{2}^{2}> 0; $
    \item [(2)] if $x_{1}=0$, then $x_{2}\ne 0 $, $\left ( \mathcal{A}_{1}\cdot x \right )_{2} \left ( \mathcal{A}_{2}\cdot x \right )_{2}=x_{2}^{6} +3x_{1}^{2}x_{2}^{4}+2x_{1}^{4}x_{2}^{2}> 0. $
\end{enumerate}
From the above analysis, $\tilde{\mathcal{A} }$ is of type $VP$. Let
$x=(-1,-1)^{T}$. Then $\left ( \mathcal{A}_{1}x^{3} \right
)_{1}=-1$, $\left ( \mathcal{A}_{2}x^{3} \right )_{1}=-5$, $\left (
\mathcal{A}_{1}x^{3} \right )_{2}=-2$, $\left ( \mathcal{A}_{2}x^{3}
\right )_{2}=-3$.
  So, $\tilde{\mathcal{A} }$ is not of type $VE$.
\end{example}

\begin{example}
    Let
    $\tilde{\mathcal{A} }$ = $\left \{ \mathcal{A}_{1}, \mathcal{A}_{2}   \right \}$, where
    \begin{gather*}
        \mathcal{A}_{1}\left ( 1, :, : \right )  =\begin{pmatrix}
            0 & 1\\-1
              &1
            \end{pmatrix},~\mathcal{A}_{1}\left ( 2, :, : \right )=\begin{pmatrix}
             1 & 0\\0
              &1
            \end{pmatrix},
        \\\mathcal{A}_{2}\left ( 1, :, : \right )  =\begin{pmatrix}
                0 & -2\\1
                  &1
                \end{pmatrix},~\mathcal{A}_{2}\left ( 2, :, : \right )=\begin{pmatrix}
                -1 & 0\\-2
                  &-1
            \end{pmatrix}.
        \end{gather*}
        By the calculation, we obtain
        \begin{equation*}
            \mathcal{A}_{1}x^{2} = \begin{pmatrix}
                x_{2}^{2}    \\x_{1}^{2} +x_{2}^{2}

                \end{pmatrix},~
            \mathcal{A}_{2}x^{2} = \begin{pmatrix}
                    x_{2}^{2}-x_{1}x_{2}    \\-x_{1}^{2}-2x_{1}x_{2} -x_{2}^{2}

                \end{pmatrix}.
            \end{equation*}
Suppose $x\ne\theta$, we discuss two cases: $x_{2}^{2}-x_{1}x_{2}>0$
and $x_{2}^{2}-x_{1}x_{2}=0$.
\begin{enumerate}
    \item [(1)] if $x_{2}^{2}-x_{1}x_{2}>0$, then $x_{2}\ne0$, $x_{2}^2>0; $
    \item [(2)] if $x_{2}^{2}-x_{1}x_{2}=0$, then $x_{2}= 0$ or $x_{2}=x_{1}\ne0$. If $x_{2}= 0$, then $x_{1}^{2} +x_{2}^{2}>0$ and $-x_{1}^{2}-2x_{1}x_{2} -x_{2}^{2}<0$. If $x_{2}=x_{1}\ne0$, then $x_{1}^{2} +x_{2}^{2}>0$ and $-x_{1}^{2}-2x_{1}x_{2} -x_{2}^{2}<0$.
\end{enumerate}
From the above analysis, the corresponding system \eqref{eq:VR} has
no solution. So, $\tilde{\mathcal{A} }$ is of type $VR_{0}$. Let
$x=(1,1)^{T}$. Then $\left ( \mathcal{A}_{1}\cdot x \right )_{1}
\left ( \mathcal{A}_{2}\cdot x \right )_{1}=0$, $\left (
\mathcal{A}_{1}\cdot x \right )_{2} \left ( \mathcal{A}_{2}\cdot x
\right )_{2}=-8$. Thus, $\tilde{\mathcal{A} }$ is not of type $VP$.
\end{example}
Examples 3.4 and 3.5 are given to illustrate these three aspects:
$\left(1\right)$ the existence of type
$VP$-$\uppercase\expandafter{\romannumeral2}$; $\left(2\right)$ type
$VP$ is not necessarily type
$VP$-$\uppercase\expandafter{\romannumeral2}$; $\left(3\right)$ type
$VP$-$\uppercase\expandafter{\romannumeral1}$ is not necessarily
type $VP$.
\begin{example}
Let $\mathcal{A}_{1}$, $\mathcal{A}_{2} \in R^{\left[3, 2\right]}$,
where
    \begin{gather*}
        \mathcal{A}_{1}\left ( 1, :, : \right )  =\begin{pmatrix}
            1 & 0\\0
              &1
            \end{pmatrix},~\mathcal{A}_{1}\left ( 2, :, : \right )=\begin{pmatrix}
             1 & 1\\-1
              &0
            \end{pmatrix},
        \\\mathcal{A}_{2}\left ( 1, :, : \right )  =\begin{pmatrix}
                1 & 1\\-1
                  &1
                \end{pmatrix},~\mathcal{A}_{2}\left ( 2, :, : \right )=\begin{pmatrix}
                0 & 0\\0
                  &-1
            \end{pmatrix}.
        \end{gather*}
        Let $\mathcal{Z} \in S^{\left[2, 2\right]}$ be an arbitrary symmetric tensor and be of form
        \begin{gather*}
            \mathcal{Z}=\begin{pmatrix}
                x_{1} & x_{2}\\x_{2}
                  &x_{3}
                \end{pmatrix},
        \end{gather*}
        where $x_{1}$ and $x_{3}$ at least one is not zero. By the calculation, we obtain
        \begin{equation*}
            \mathcal{A}_{1}\mathcal{Z} = \begin{pmatrix}
                x_{1}+x_{3}    \\x_{1}

                \end{pmatrix},~
            \mathcal{A}_{2}\mathcal{Z} = \begin{pmatrix}
                x_{1}+x_{3}     \\-x_{3}

                \end{pmatrix}.
            \end{equation*}
            Obviously, at least one of $\left(x_{1}+x_{3}\right)^{2}$ and $-x_{1}x_{3}$ is greater than zero. Thus, $\tilde{\mathcal{A} }= \left \{ \mathcal{A}_{1}, \mathcal{A}_{2}   \right \}$ is of type $VP$-$\uppercase\expandafter{\romannumeral2}$.
\end{example}
In Example \ref{ex:3.2}, let
\begin{gather*}
    \mathcal{Z}\left ( 1, :, : \right )=\begin{pmatrix}
        0 & -1\\-1
          &0
        \end{pmatrix},~
        \mathcal{Z}\left ( 2, :, : \right )=\begin{pmatrix}
            -1 & 0\\0
              &1
            \end{pmatrix}.
\end{gather*}
Then $\left ( \mathcal{A}_{1} \mathcal{Z} \right )_{i} \left ( \mathcal{A}_{2} \mathcal{Z} \right )_{i} =0 $ for all $i\in \left\{1, 2\right\}$. Thus, $\tilde{\mathcal{A} }$ = $\left \{ \mathcal{A}_{1}, \mathcal{A}_{2}   \right \}$ in Example \ref{ex:3.2} is not of type $VP$-$\uppercase\expandafter{\romannumeral2}$.
\begin{example}
    Let $\mathcal{A}_{1}$, $\mathcal{A}_{2} \in R^{\left[3, 2\right]}$, where
    \begin{gather*}
        \mathcal{A}_{1}\left ( 1, :, : \right )  =\begin{pmatrix}
            1 & 0\\0
              &0
            \end{pmatrix},~\mathcal{A}_{1}\left ( 2, :, : \right )=\begin{pmatrix}
             1 & 0\\0
              &1
            \end{pmatrix},
        \\\mathcal{A}_{2}\left ( 1, :, : \right )  =\begin{pmatrix}
                0 & 2\\-1
                  &0
                \end{pmatrix},~\mathcal{A}_{2}\left ( 2, :, : \right )=\begin{pmatrix}
                1 & 3\\0
                  &1
            \end{pmatrix}.
        \end{gather*}
        By the calculation, we obtain
        \begin{equation*}
            \mathcal{A}_{1}x^{2} = \begin{pmatrix}
                x_{2}^{2}    \\x_{1}^{2} +x_{2}^{2}

                \end{pmatrix},~
            \mathcal{A}_{2}x^{2} = \begin{pmatrix}
                    x_{1}x_{2}    \\x_{1}^{2}+3x_{1}x_{2}+x_{2}^2

                \end{pmatrix}.
            \end{equation*}
Obviously, $\mathcal{A}_{1}\cdot x \wedge \mathcal{A}_{2}\cdot
x\le\theta\le \mathcal{A}_{1}\cdot x \vee  \mathcal{A}_{2}\cdot x $
have no solution in $\left ( R_{+}^{n}/\left\{\theta \right\}
\right)     \cup  \left ( R_{-}^{n}/\left\{\theta \right\} \right
)$. So, $\tilde{\mathcal{A} } = \left \{ \mathcal{A}_{1},
\mathcal{A}_{2} \right \}$ is of type
$VP$-$\uppercase\expandafter{\romannumeral1}$. Let $x=( 1,-1)^{T}$.
Then $\left ( \mathcal{A}_{1}\cdot x \right )_{1} \left (
\mathcal{A}_{2}\cdot x \right )_{1}=-1$, $\left (
\mathcal{A}_{1}\cdot x \right )_{2} \left ( \mathcal{A}_{2}\cdot x
\right )_{2}=-2$. Thus, $\tilde{\mathcal{A} } = \left \{
\mathcal{A}_{1}, \mathcal{A}_{2}   \right \}$ is not of type $VP$.
\end{example}
Next, a special tensor set is introduced. Inspired by the
$P$-function \cite{kanzow1996equivalence}, the following condition
is given

\begin{equation}\label{condition:1}
    \underset{i\in \left [ n \right ] }{\text{max}} \left ( G\left ( x \right )-G\left ( y \right )   \right )_{i} \left ( F\left ( x \right )-F\left ( y \right )   \right )_{i}>0
\end{equation}
for all $x,y\in R^{n}$ and $x\ne y$,
where $G\left ( x \right )$, $F\left ( x \right )$ are two functions. Now, we define a new tensor set, see Definition \ref{D:3.2}.
\begin{definition}\label{D:3.2}
We say that $\tilde{\mathcal{A} }=\left \{ \mathcal{A}_{1},
\mathcal{A}_{2}   \right \}$ is of type strong $VP$ if the mappings
$G\left(x\right):=\mathcal{A}_{1} x^{m-1}$ and
$F\left(x\right):=\mathcal{A}_{2} x^{m-1}$ satisfy the condition \eqref{condition:1} for all
$x,y\in R^{n}$ and $x\ne y$.
\end{definition}
\begin{proposition}
There is no type strong $VP$ tensor set with odd order.
\end{proposition}
\begin{proof}
Assume that $\tilde{\mathcal{A} }$ = $\left \{ \mathcal{A}_{1},
\mathcal{A}_{2}   \right \}$ is of type strong $VP$ with
$\mathcal{A}_{1}, \mathcal{A}_{2}  \in R^{\left [ m, n  \right ] }$,
and $m$ is an odd number. Obviously,
    \begin{equation*}
        \mathcal{A}_{1}\left(x^{m-1}-\left(-x\right)^{m-1}\right)=\mathcal{A}_{2}\left(x^{m-1}-\left(-x\right)^{m-1}\right)=\theta~for~any~x\ne \theta.
    \end{equation*}
    This contradicts the assumption.
\end{proof}
\begin{proposition}\label{pro:3.3}
type $VP$-$\uppercase\expandafter{\romannumeral2}$ with even order
$\subseteq$ type strong $VP$ $\subseteq$ type
$VP$.
\end{proposition}
Since there is no type strong $VP$ tensor set with odd order,
$\tilde{\mathcal{A} }$ in Example \ref{ex:3.1} is of type $VP$, but
not type strong $VP$. Examples 3.6 and 3.7 are given to illustrate
three aspects: $\left(1\right)$ the existence of type strong $VP$;
$\left(2\right)$ type strong $VP$ is not necessarily type
$VP$-$\uppercase\expandafter{\romannumeral2}$ with even order;
$\left(3\right)$ type $VP$ with even order is not necessarily type
strong $VP$.
\begin{example}\label{ex:sp}
    Let
    $\tilde{\mathcal{A} } = \left \{ \mathcal{A}_{1}, \mathcal{A}_{2}   \right \}$, where
    \begin{gather*}
        \mathcal{A}_{1}\left ( 1, 1, :, : \right )  =\begin{pmatrix}
            1  & 1\\-2
              &1
            \end{pmatrix},~\mathcal{A}_{1}\left (1, 2, :, : \right )=\begin{pmatrix}
             1 & -1\\0
              &0
            \end{pmatrix},
        \\\mathcal{A}_{1}\left ( 2, 1, :, : \right )  =\begin{pmatrix}
            0  & 0\\1
              &0
            \end{pmatrix},~\mathcal{A}_{1}\left (2, 2, :, : \right )=\begin{pmatrix}
             0 & -1\\1
              &1
            \end{pmatrix},
        \\\mathcal{A}_{2}\left ( 1, 1, :, : \right )  =\begin{pmatrix}
            1  & 0\\0
              &0
            \end{pmatrix},~\mathcal{A}_{2}\left (1, 2, :, : \right )=\begin{pmatrix}
             0 & 0\\0
              &0
            \end{pmatrix},
        \\\mathcal{A}_{2}\left ( 2, 1, :, : \right )  =\begin{pmatrix}
            0  & -1\\1
              &0
            \end{pmatrix},~\mathcal{A}_{2}\left (2, 2, :, : \right )=\begin{pmatrix}
             0 & 0\\0
              &1
            \end{pmatrix}.
    \end{gather*}
    By the calculation, we obtain
        \begin{equation*}
            \mathcal{A}_{1}x^{3}-\mathcal{A}_{1}y^{3} = \begin{pmatrix}
                x_{1}^{3}-y_{1}^{3}    \\x_{1}^{2}x_{2} -y_{1}^{2}y_{2}+x_{2}^{3}-y_{2}^{3}

                \end{pmatrix},~
                \mathcal{A}_{2}x^{3}-\mathcal{A}_{2}y^{3} = \begin{pmatrix}
                    x_{1}^{3}-y_{1}^{3}   \\x_{2}^{3}-y_{2}^{3}

                \end{pmatrix}.
            \end{equation*}
            Suppose $x\ne y$, we discuss two cases: $x_{1}\ne y_{1}$ and $x_{1}= y_{1}$.
\begin{enumerate}
    \item [(1)] if $x_{1}\ne y_{1}$, then $\left ( \mathcal{A}_{1}x^{3}-\mathcal{A}_{1}y^{3} \right )_{1} \left ( \mathcal{A}_{2}x^{3}-\mathcal{A}_{2}y^{3}\right )_{1}=\left(x_{1}^{3}-y_{1}^{3}\right)^{2}> 0; $
    \item [(2)] if $x_{1}= y_{1}$, then $x_{2}\ne y_{2}$ and $\left ( \mathcal{A}_{1}x^{3}-\mathcal{A}_{1}y^{3} \right )_{2} \left ( \mathcal{A}_{2}x^{3}-\mathcal{A}_{2}y^{3}\right )_{2}=\left(x_{1}^{2}x_{2} -y_{1}^{2}y_{2}\right)\left(x_{2}^{3}-y_{2}^{3}\right)+\left(x_{2}^{3}-y_{2}^{3}\right)^{2}> 0. $
\end{enumerate}
From the above analysis, $\tilde{\mathcal{A} }$ = $\left \{ \mathcal{A}_{1}, \mathcal{A}_{2}   \right \}$ is of type strong $VP$. Let
\begin{equation*}
    \mathcal{Z}\left ( 1, :, : \right )  =\begin{pmatrix}
        0 & -1\\-1
          &0
        \end{pmatrix},~\mathcal{Z}\left ( 2, :, : \right )=\begin{pmatrix}
         -1 & 0\\0
          &1
        \end{pmatrix}.
\end{equation*}
Then $\left ( \mathcal{A}_{1}\mathcal{Z} \right)_{1}\left (
\mathcal{A}_{2}\mathcal{Z}\right )_{1}=0$, $\left (
\mathcal{A}_{1}\mathcal{Z} \right)_{2}\left (
\mathcal{A}_{2}\mathcal{Z}\right )_{2}=0$. Thus, $\tilde{\mathcal{A}
} = \left \{ \mathcal{A}_{1}, \mathcal{A}_{2}   \right \}$ is not
type $VP$-$\uppercase\expandafter{\romannumeral2}$ with even order.
\end{example}
\begin{example}
    Let
    $\tilde{\mathcal{A} } = \left \{ \mathcal{A}_{1}, \mathcal{A}_{2}   \right \}$, where
    \begin{gather*}
        \mathcal{A}_{1}\left ( 1, 1, :, : \right )  =\begin{pmatrix}
            0  & -1\\2
              &0
            \end{pmatrix},~\mathcal{A}_{1}\left (1, 2, :, : \right )=\begin{pmatrix}
             0 & 0\\0
              &0
            \end{pmatrix},
        \\\mathcal{A}_{1}\left ( 2, 1, :, : \right )  =\begin{pmatrix}
            1 & 0\\0
              &0
            \end{pmatrix},~\mathcal{A}_{1}\left (2, 2, :, : \right )=\begin{pmatrix}
             0 & 0\\0
              &1
            \end{pmatrix},
        \\\mathcal{A}_{2}\left ( 1, 1, :, : \right )  =\begin{pmatrix}
            0  & 1\\1
              &0
            \end{pmatrix},~\mathcal{A}_{2}\left (1, 2, :, : \right )=\begin{pmatrix}
             -1 & 0\\0
              &0
            \end{pmatrix},
        \\\mathcal{A}_{2}\left ( 2, 1, :, : \right )  =\begin{pmatrix}
            1 & -1\\1
              &0
            \end{pmatrix},~\mathcal{A}_{2}\left (2, 2, :, : \right )=\begin{pmatrix}
             0 & 0\\0
              &1
            \end{pmatrix}.
    \end{gather*}
    By the calculation, we obtain
        \begin{equation*}
            \mathcal{A}_{1}x^{3} = \begin{pmatrix}
                x_{1}^{2}x_{2}    \\x_{1}^{3} +x_{2}^{3}

                \end{pmatrix},~
            \mathcal{A}_{2}x^{3} = \begin{pmatrix}
                x_{1}^{2}x_{2}    \\x_{1}^{3} +x_{2}^{3}

                \end{pmatrix}.
            \end{equation*}
Suppose $x\ne\theta$, we discuss two cases: $x_{1}x_{2}\ne 0$ and
$x_{1}x_{2}= 0$.
\begin{enumerate}
    \item [(1)] if $x_{1}x_{2}\ne 0$, then $\left ( \mathcal{A}_{1}\cdot x \right )_{1} \left ( \mathcal{A}_{2}\cdot x \right )_{1}=x_{1}^{4}x_{2}^{2}>0 $;
    \item [(2)] if $x_{1}x_{2}= 0$, then  $\left ( \mathcal{A}_{1}\cdot x \right )_{2} \left ( \mathcal{A}_{2}\cdot x \right )_{2}=\left(x_{1}^{3} +x_{2}^{3}\right)^{2}>0 $.
\end{enumerate}
From the above analysis, $\tilde{\mathcal{A} }$ = $\left \{
\mathcal{A}_{1}, \mathcal{A}_{2}  \right \}$ is of type $VP$ with
even order. Let $x=( 0,1)^{T}$, $y=(1,0)^{T}$. Then
\[ \left (
\mathcal{A}_{1}x^{3}-\mathcal{A}_{1}y^{3} \right )_{1} \left (
\mathcal{A}_{2}x^{3}-\mathcal{A}_{2}y^{3}\right )_{1}= 0 ~and~\left
( \mathcal{A}_{1}x^{3}-\mathcal{A}_{1}y^{3} \right )_{2} \left (
\mathcal{A}_{2}x^{3}-\mathcal{A}_{2}y^{3}\right )_{2}=0.
\]
Thus, $\tilde{\mathcal{A}} =\left \{ \mathcal{A}_{1},
\mathcal{A}_{2}   \right \}$ is not of type strong $VP$.
\end{example}
\section{Solution of the VTCP}
In this section, we show that the solution set of the
VTCP$(\tilde{\mathcal{A}}, \tilde{q })$ is bounded under certain
conditions, and some sufficient conditions for the existence and
uniqueness of the solution of the VTCP$(\tilde{\mathcal{A}}, \tilde{q
})$ are obtained from the view of the degree theory and the equal
form of the minimum function.

To show that  the solution set of the VTCP$(\tilde{\mathcal{A}},
\tilde{q })$ is bounded under certain conditions, we give the
following result, see Theorem 4.1.

\begin{theorem}\label{thm: bound}
$\tilde{\mathcal{A} }$ is of type $VR_{0}$ if and only if  the solution
set of the VTCP$(\tilde{\mathcal{A}}, \tilde{q })$ for all $q_{1},
q_{2}\in R^{n} $ is bounded.
\begin{proof}
It is now proved that the following three conditions are equivalent:
\begin{enumerate}
            \item [$(i)$] $\tilde{\mathcal{A} }$ is of type $VR_{0}$;
            \item [$(ii)$] $\forall q_{1}, q_{2}\in R^{n}$, $t>0$, $s\in R$, the
\begin{align*}
        {} & \tau\left ( q_{1}, q_{2}, s, t \right )
        \\
={} & \left \{ x:q_{1}+\mathcal{A}_{1} x^{m-1} \ge\theta,
q_{2}+\mathcal{A}_{2} x^{m-1} \ge\theta , \left (
q_{1}+\mathcal{A}_{1} x^{m-1} \right )^{T} \left (
q_{2}+t\mathcal{A}_{2} x^{m-1} \right )\le s\right \}
\end{align*}
            % \begin{multline*}
            %   \tau\left ( q_{1}, q_{2}, s, t \right )=\\\left \{ x:q_{1}+\mathcal{A}_{1} x^{m-1} \ge\theta, q_{2}+\mathcal{A}_{2} x^{m-1} \ge\theta , \left ( q_{1}+\mathcal{A}_{1} x^{m-1} \right )^{T} \left ( q_{2}+t\mathcal{A}_{2} x^{m-1} \right )\le s\right \}
            % \end{multline*}
            is bounded;
            \item [$(iii)$] The solution set of the VTCP$(\tilde{\mathcal{A} }, \tilde{q })$ for all $q_{1}, q_{2}\in R^{n} $ is bounded.

            $(i)\Rightarrow (ii)$: Assume that there exist $q_{1}^{*}, q_{2}^{*}\in R^{n}$, $t^{*}>0$ and $s^{*}\in R$ such that the set $\tau\left ( q_{1}^{*}, q_{2}^{*}, s^{*}, t^{*} \right )$ is not bounded.
            Then  there is an unbounded sequence $\left \{ x^{k} \right \}\in \tau\left ( q_{1}^{*}, q_{2}^{*}, s^{*}, t^{*} \right )$.
            Since the sequence $\{\frac{x^{k}}{\| x^{k}\|}\}$ is bounded, there is a convergence subsequence $\{\frac{x^{k_{j}}}{\|x^{k_{j}}\|}\}$ such that when $j\rightarrow\infty$,
            $\frac{x^{k_{j}}}{\left \| x^{k_{j}} \right \| } \rightarrow x_{1}\ne\theta$ and $\left \| x^{k_{j}} \right \|\rightarrow \infty$. Thus, it follows that
            \begin{gather*}
                \mathcal{A}_{1}\! \left (\!\frac{x^{k^{j}}}{\left \|  x^{k^{j}}\right \| } \! \right )^{m-1}\!+\!\frac{q_{1}^{*}}{\left ( \!\left \|  x^{k^{j}}\right \| \!\right )^{m-1} }\ge \theta ,
                \mathcal{A}_{2} \left (\!\frac{x^{k^{j}}}{\left \|  x^{k^{j}}\right \| }  \!\right )^{m-1}\!+\!\frac{q_{2}^{*}}{\left ( \!\left \|  x^{k^{j}}\right \|\! \right )^{m-1} }\ge \theta ,
                \\\left [ \mathcal{A}_{1} \left (\! \frac{x^{k^{j}}}{\left \|  x^{k^{j}}\right \| } \! \right )^{m-1}\!+\!\frac{q_{1}^{*}}{\left (\! \left \|  x^{k^{j}}\right \| \!\right )^{m-1} }  \right ]^{T}\left [ t^{*}\mathcal{A}_{2} \left (\! \frac{x^{k^{j}}}{\left \|  x^{k^{j}}\right \| }   \!\right )^{m-1}\!+\!\frac{q_{2}^{*}}{\left ( \left \|  x^{k^{j}}\right \| \right )^{m-1} }  \right ]\le \frac{s^{*}}{\left ( \left \|  x^{k^{j}}\right \| \right )^{2m-2} }.
            \end{gather*}

    Let $j\rightarrow  +\infty$. Then
      \begin{equation*}
        \mathcal{A}_{1} \left (x_{1}   \right )^{m-1}\ge \theta, \mathcal{A}_{2} \left (x_{1}   \right )^{m-1}\ge \theta, \left [ \mathcal{A}_{1} \left (x_{1}   \right )^{m-1} \right ]^{T}\left [ t^{*}\mathcal{A}_{2} \left (x_{1}   \right )^{m-1} \right ]\le \theta.
      \end{equation*}
    Thus, $\left[ \mathcal{A}_{1} \left (x_{1}   \right )^{m-1} \right ]^{T}\left[\mathcal{A}_{2} \left (x_{1}   \right )^{m-1} \right ]= \theta$. This contradicts that $\tilde{\mathcal{A} }$ is of type $VR_{0}$.

        $(ii)\Rightarrow (iii)$: Let $s=0$, $t=1$. Then the set
        \begin{align*}
            {} & \tau\left ( q_{1}, q_{2}, 0, 1 \right )
            \\
            ={} & \left \{ x:q_{1}+\mathcal{A}_{1} x^{m-1} \ge0, q_{2}+\mathcal{A}_{2} x^{m-1} \ge0 , \left ( q_{1}+\mathcal{A}_{1} x^{m-1} \right )^{T} \left ( q_{2}+\mathcal{A}_{2} x^{m-1} \right )\le 0\right \}
            \\
            ={} & \left \{ x:q_{1}+\mathcal{A}_{1} x^{m-1} \ge0, q_{2}+\mathcal{A}_{2} x^{m-1} \ge0 , \left ( q_{1}+\mathcal{A}_{1} x^{m-1} \right )^{T} \left ( q_{2}+\mathcal{A}_{2} x^{m-1} \right )= 0\right \}
          \end{align*}
        is bounded, which implies that the solution set of the VTCP$( \tilde{\mathcal{A} }, \tilde{q} )$ is bounded.

        $(iii)\Rightarrow (i)$: Assume that $\tilde{\mathcal{A} }$ is not of type $VR_{0}$. There is $x\ne\theta$ such that $\mathcal{A}_{1}\cdot x \wedge \mathcal{A}_{2}\cdot x=\theta$. Hence, $kx\in\tau\left ( \theta, \theta, 0, 1 \right )$ for all $k\ge\theta$. This contradicts the condition $(iii)$.
      \end{enumerate}
By the above discussion, the conclusion in Theorem \ref{thm: bound}
holds.
     \end{proof}
\end{theorem}
\begin{theorem}
    $\tilde{\mathcal{A} }$ is of type $VE$ if and only if the VTCP$(\tilde{\mathcal{A} }, \tilde{q } )$ has at most $\theta$ solution for all $q_{1}\ge\theta$, $q_{2}\ge \theta$.
\end{theorem}
\begin{proof}
Sufficiency. Assume that $\tilde{\mathcal{A} }$ is not of type $VE$,
i.e., there is a vector $x\ne\theta$ such that $\mathcal{A}_{1}\cdot
x \wedge \mathcal{A}_{2}\cdot x\le \theta$. If $\mathcal{A}_{1}\cdot
x \wedge \mathcal{A}_{2}\cdot x=\theta$, then it contradicts that
the VTCP$(\tilde{\mathcal{A}}, \tilde{q })$ has at most $\theta$
solution for all $q_{1}\ge\theta$, $q_{2}\ge \theta$. If
$\mathcal{A}_{1}\cdot x \wedge \mathcal{A}_{2}\cdot x<\theta$, then
we take
        \begin{gather*}
            q_{1}=\begin{cases}
                -\left ( \mathcal{A}_{1}\cdot x \right )_{i}, \text{ if } \left( \mathcal{A}_{1}\cdot x \right )_{i}\le \left( \mathcal{A}_{2}\cdot x \right )_{i},  \\
                \left | \left ( \mathcal{A}_{1}\cdot x \right )_{i} \right | , \text{ if } \left( \mathcal{A}_{1}\cdot x \right )_{i}> \left( \mathcal{A}_{2}\cdot x \right )_{i}
                , \end{cases}
            \\q_{2}=\begin{cases}
                -\left ( \mathcal{A}_{2}\cdot x \right )_{i}, \text{ if } \left( \mathcal{A}_{1}\cdot x \right )_{i}\ge  \left( \mathcal{A}_{2}\cdot x \right )_{i},  \\
                \left | \left ( \mathcal{A}_{2}\cdot x \right )_{i} \right | , \text{ if } \left( \mathcal{A}_{1}\cdot x \right )_{i}< \left( \mathcal{A}_{2}\cdot x \right )_{i}
                . \end{cases}
        \end{gather*}
Obviously, $q_{1},q_{2}\ge \theta$, and $x$ is a solution of the VTCP$(
\tilde{\mathcal{A} }, \tilde{q })$. This contradicts that the VTCP$(
\tilde{\mathcal{A} }, \tilde{q })$ has at most $\theta$ solution for
all $q_{1}\ge\theta$, $q_{2}\ge \theta$. So, $\tilde{\mathcal{A} }$
is of type $VE$.

Necessity. Obviously, $\theta$ is the solution of the VTCP$(
\tilde{\mathcal{A} }, \tilde{q })$ for some $q_{1}$, $q_{2}\ge
\theta$. Assume that there exist $q_{1}$, $q_{2}\ge \theta$ such
that the VTCP$( \tilde{\mathcal{A} }, \tilde{q })$ has a nonzero
solution $x$. Since $q_{1}, q_{2}\ge \theta$ and $x$ is a nonzero
solution of the VTCP$( \tilde{\mathcal{A} }, \tilde{q })$, we obtain
\begin{equation*}
            \mathcal{A}_{1}\cdot x \wedge \mathcal{A}_{2}\cdot x\le \left(\mathcal{A}_{1}\cdot x +q_{1}\right)\wedge \left(\mathcal{A}_{2}\cdot x+q_{2}\right)=\theta.
        \end{equation*}
This contradicts that $\tilde{\mathcal{A} }$ is of type $VE$ tensor.
So, the VTCP$(\tilde{\mathcal{A} }, \tilde{q })$ has at most $\theta$
solution for all $q_{1}$, $q_{2}\ge \theta$.
\end{proof}

Next, from the view of the degree theory, we discuss the existence
of the solution of the VTCP$( \tilde{\mathcal{A} }, \tilde{q })$.

Here is a brief review of degree theory.
Suppose that $\Omega$ is a bounded open set in $R^{n}$. $ \partial
\Omega $ and $ \overline{\Omega }$  represent the boundary of
$\Omega $ and the set $\partial \Omega \cup \Omega$, respectively. Let the
continuous function $f:\overline{\Omega }\longrightarrow R^{n}$ be
given and $p\notin f\left(\partial \Omega\right)$. Then the degree
of $f=p$ on $\Omega$ can be well defined and denoted as $deg\left (
f, \Omega, p\right )$, which is an integer and implies the number of solutions for $f=p$,
see \cite{facchinei2003finite}.

Let $\tilde{\mathcal{A} }$ be of type $VR_{0}$,
\begin{equation*}
F_{\left ( \tilde{\mathcal{A} }  \right )}=\mathcal{A}_{1} x^{m-1}\wedge  \mathcal{A}_{2} x^{m-1}.
\end{equation*}
Then $F_{\left ( \tilde{\mathcal{A} } \right )}=\theta\Rightarrow  x=\theta$. Let $\Omega$  be a bounded open set that contains $\theta$. Then $deg\left ( F_{\left ( \tilde{\mathcal{A} } \right )}, \Omega, \theta \right )$ is defined. Furthermore, $deg\left ( F_{\left ( \tilde{\mathcal{A} } \right )}, \Omega, \theta \right )$ is independent of the bounded open set $\Omega$. We denote it by VTCP-$deg\tilde{\mathcal{A} }$.

According to the degree theory, the following results can be obtained.
\begin{theorem}
    If $\tilde{\mathcal{A} }$ is of type $VR_{0}$ and VTCP-$deg\tilde{\mathcal{A} }$ is nonzero, then $\tilde{\mathcal{A} }$ is of type $VQ$.
    \begin{proof}
        Let $q_{1}, q_{2}\in R^{n}$ and
        \[
            F\left ( x, t \right )=\left(tq_{1}+\mathcal{A}_{1} x^{m-1}\right)\wedge \left(tq_{2}+\mathcal{A}_{2} x^{m-1}\right).
        \]
            By Theorem \ref{thm: bound},  the set
            \begin{equation*}
                X=\left \{x:  F\left ( x, t \right )=\theta, t\in \left[0, 1\right] \right \}
            \end{equation*}
            is bounded. Hence, there is a bounded open set  $\Omega$ containing the set $X$ and $X\cap \partial \Omega =\emptyset$.
            Obviously, $F\left ( x, t \right )$ is a homotopy connecting the mappings $F\left ( x, 1  \right) $ and  $F\left ( x \right )$, where $F\left ( x \right )=F\left ( x, 0 \right )$.
            By the homotopy invariance property of the degree \cite{1978Degree}, we obtain
            \begin{equation*}
            deg\left ( F\left ( x, 1\right ), \Omega, \theta \right )=deg\left ( F\left ( x \right ), \Omega, \theta \right )={\rm VTCP}\text{-}deg\tilde{\mathcal{A}}.
            \end{equation*}
Since VTCP-$deg\tilde{\mathcal{A} }$ is nonzero, together with the
property of the degree \cite{1978Degree}, there is a solution of
$F\left ( x, 1 \right )=\theta$ in $\Omega$, which implies that $\tilde{\mathcal{A}
}$ is of type $VQ$.
    \end{proof}
\end{theorem}

\begin{theorem}\label{the: ex}
If $\tilde{\mathcal{A} }= \left \{ \mathcal{A}_{1}, \mathcal{A}_{2}
\right \}$ is of type $VR_{0}$ with $\mathcal{A}_{1},
\mathcal{A}_{2}  \in R^{\left [ m, n  \right ] }$, $m>2$ and
$\mathcal{A}_{2}$ is a $R$-tensor, then $\tilde{\mathcal{A} }$ is of
type $VQ$.
\end{theorem}
\begin{proof}
Let
\begin{equation*}
\tau\left ( q_{1}, q_{2}, s, d, a, t \right )=\left \{ x:F\left(x\right)
\ge\theta, G\left(x\right) \ge\theta , \left(F\left(x\right)\right) ^{T} G\left(x\right) \le s\right \},
\end{equation*}
where
\begin{equation*}
F\left(x\right)=tq_{1}+\left(1-t\right)x+t\mathcal{A}_{1} x^{m-1} ~ and ~
G\left(x\right)=tq_{2}+\left(1-t\right)d+a\mathcal{A}_{2} x^{m-1}.
\end{equation*}
Similar to the proof of Theorem \ref{thm: bound}, $\tau\left (
q_{1}, q_{2}, s, d, a, t \right )$ is bounded for any $q_{1}, q_{2}
\in R^{n}$, $s\in R$, $t\in \left(0,1\right]$, $a>0$ and
$d>\theta$. Thus, $\tau\left ( q_{1}, q_{2}, 0, I, 1, t \right )$ is
bounded for any $q_{1}, q_{2}
\in R^{n}$ and $t\in \left(0,1\right]$, which is
equivalent to that the set of solutions of
\begin{equation}\label{eq:4.1}
\left(tq_{1}+\left(1-t\right)x+t\mathcal{A}_{1}x^{m-1}\right)\wedge
\left(tq_{2}+\left(1-t\right)I+\mathcal{A}_{2} x^{m-1}\right)=\theta
\end{equation}
is bounded for any $q_{1}, q_{2}
\in R^{n}$ and $t\in \left(0,1\right]$. Since
$\mathcal{A}_{2}$ is a $R$-tensor, \eqref{eq:4.1} has only one $\theta$
solution for any $q_{1}, q_{2}\in R^{n}$ when $t=0$.
Let
\begin{equation*}
 H\left(x, t\right)=\left(tq_{1}+\left(1-t\right)x+t\mathcal{A}_{1}x^{m-1}\right)\wedge \left(tq_{2}+\left(1-t\right)I+\mathcal{A}_{2} x^{m-1}\right),
\end{equation*}
where $q_{1}$, $q_{2}\in R^{n}$ are any given vectors.
Based on the above analysis, the set $X=\left \{x:  H\left(x,
t\right)=\theta, t\in \left[0,1\right]\right \}$ is bounded in
$R^{n}$. So, there is a bounded open set  $\Omega$ containing the
set $X$ and $X\cap \partial \Omega =\emptyset$. Since
$\mathcal{A}_{2}$ is a $R$-tensor, $H\left(x,0\right)=\theta$ has
only $\theta$ solution. Note that $I+\mathcal{A}_{2} x^{m-1}>\theta$
in the field around the $\theta$ vector, which implies that $H\left(x,
0\right)=x$ in the field around the $\theta$ vector.
So, $degree\left(H\left ( x, 0 \right ), \Omega, \theta\right)=1$.
Obviously, $H\left(x, t\right)$
is a homotopy connecting the mappings
\begin{equation*}
H\left(x, 1\right)=\left(q_{1}+\mathcal{A}_{1}x^{m-1}\right)\wedge
\left(q_{2}+\mathcal{A}_{2} x^{m-1}\right) ~ and ~
 H\left(x,0\right)=x\wedge \left(I+\mathcal{A}_{2} x^{m-1}\right).
 \end{equation*}
By the homotopy invariance property of the degree \cite{1978Degree},
\begin{equation*}
degree\left(H\left ( x, 1 \right ), \Omega,
\theta\right)=degree\left(H\left ( x, 0 \right ), \Omega,
\theta\right)=1.
\end{equation*}
Thus, by the property of the degree \cite{1978Degree},
$H\left ( x, 1 \right )=\theta$ has a solution, which implies that the result in Theorem \ref{the: ex} holds.
\end{proof}

Next, some uniqueness conditions of the VTCP$(\tilde{\mathcal{A} }, \tilde{q})$ are obtained.
\begin{theorem}\label{the: un}
If $\tilde{\mathcal{A} }$ is of type
$VP$-$\uppercase\expandafter{\romannumeral2}$ with even order, then
the VTCP$(\tilde{\mathcal{A} }, \tilde{q})$  has at most one
solution for all $q_{1}, q_{2} \in R^{n}$. \end{theorem}
\begin{proof}
Assume that $x$ and $y$ are the two different solutions of the
VTCP$(\tilde{\mathcal{A} }, \tilde{q})$. Then
           \begin{equation}
            \left(q_{1}+\mathcal{A}_{1} x^{m-1}\right)\wedge \left(q_{2}+\mathcal{A}_{2} x^{m-1}\right)=\theta \label{eq:1}
            \end{equation}
            and
            \begin{equation}
                \left(q_{1}+\mathcal{A}_{1} y^{m-1}\right)\wedge \left(q_{2}+\mathcal{A}_{2} y^{m-1}\right)=\theta. \label{eq:2}
            \end{equation}
            From \eqref{eq:1} and \eqref{eq:2}, we obtain
            \begin{equation}
            \mathcal{A}_{1} \left(x^{m-1}-y^{m-1}\right)\wedge \mathcal{A}_{2} \left(x^{m-1}-y^{m-1}\right)\le\theta \label{eq:3}
            \end{equation}
            and
            \begin{equation}
                \mathcal{A}_{1} \left(y^{m-1}-x^{m-1}\right)\wedge \mathcal{A}_{2} \left(y^{m-1}-x^{m-1}\right)\le\theta. \label{eq:4}
            \end{equation}
            Note that \eqref{eq:4} is equivalent to that
            \begin{equation}
                \mathcal{A}_{1} \left(x^{m-1}-y^{m-1}\right)\vee \mathcal{A}_{2} \left(x^{m-1}-y^{m-1}\right)\ge\theta. \label{eq:5}
                \end{equation}
From \eqref{eq:3} and \eqref{eq:5}, $z \in S^{\left[m-1,n\right]}$
satisfies
\begin{equation}
\mathcal{A}_{1}  \mathcal{Z}\wedge \mathcal{A}_{2}  \mathcal{Z} \le
\theta \le \mathcal{A}_{1}  \mathcal{Z} \vee
\mathcal{A}_{2}\mathcal{Z},\label{eq:3.7}
\end{equation}
where $\mathcal{Z}=x^{m-1}-y^{m-1}$. Since the order is even, the
principal diagonal elements of $\mathcal{Z}$ are not all zeros. This
contradicts that $\tilde{\mathcal{A} }$ is of type
$VP$-$\uppercase\expandafter{\romannumeral2}$. So, the
VTCP$(\tilde{\mathcal{A} }, \tilde{q})$  has at most one solution.
\end{proof}

\begin{theorem}
Let $\tilde{\mathcal{A} }= \left \{ \mathcal{A}_{1}, \mathcal{A}_{2}
\right \}$ with $\mathcal{A}_{1}, \mathcal{A}_{2}  \in R^{\left [ m,
n  \right ] }$ and $m$ be an even number. If the
VTCP$(\tilde{\mathcal{A} }, \tilde{q})$ for all $q_{1}, q_{2} \in
R^{n}$ has at most one solution, then $\tilde{\mathcal{A} }$ is of
type $VP$-$\uppercase\expandafter{\romannumeral1}$.
\end{theorem}
\begin{proof}
Assume that $\tilde{\mathcal{A} }= \left \{ \mathcal{A}_{1},
\mathcal{A}_{2} \right \}$ is not of type
$VP$-$\uppercase\expandafter{\romannumeral1}$. Then there is a vector $x\in\left ( R_{+}^{n}/\left\{\theta
\right\}\right ) \cup  \left ( R_{-}^{n}/\left\{\theta \right\}
\right )$ such that $\mathcal{A}_{1}\cdot x \wedge
\mathcal{A}_{2}\cdot x\le\theta\le \mathcal{A}_{1}\cdot x \vee
\mathcal{A}_{2}\cdot x $. Thus,
    \begin{equation}
        \left ( \mathcal{A}_{1}\cdot x \right )_{+}  \wedge \left ( \mathcal{A}_{2}\cdot x \right )_{+} =\left ( \mathcal{A}_{1}\cdot x \right )_{-}  \wedge \left ( \mathcal{A}_{2}\cdot x \right )_{-}=\theta. \label{eq:6}
    \end{equation}
    Let
    \begin{equation*}
        q_{1}=\left ( \mathcal{A}_{1}\cdot x \right )_{+}  - \mathcal{A}_{1} \cdot\left ( x \right )_{+}
    ~and~
        q_{2}=\left ( \mathcal{A}_{2}\cdot x \right )_{+}  - \mathcal{A}_{2} \cdot\left ( x \right )_{+}.
    \end{equation*}
    Then, we need to prove
     \begin{equation}
        q_{1}=\left ( \mathcal{A}_{1}\cdot x \right )_{+}  - \mathcal{A}_{1} \cdot\left ( x \right )_{+}=\left ( \mathcal{A}_{1}\cdot x \right )_{-}  - \mathcal{A}_{1} \cdot\left ( x \right )_{-} \label{eq:7}
    \end{equation}
    and
    \begin{equation}
        q_{2}=\left ( \mathcal{A}_{2}\cdot x \right )_{+}  - \mathcal{A}_{2} \cdot\left ( x \right )_{+}=\left ( \mathcal{A}_{2}\cdot x \right )_{-}  - \mathcal{A}_{2} \cdot\left ( x \right )_{-}. \label{eq:8}
    \end{equation}
    Here, we only prove \eqref{eq:7} holds. We consider two cases: $x\in R_{+}^{n}/\left\{\theta \right\}$ and $x\in R_{-}^{n}/\left\{\theta \right\}$.
    \begin{enumerate}
        \item [(1)] if $x\in R_{+}^{n}/\left\{\theta \right\}$, then
        \[\left ( \mathcal{A}_{1}\cdot x \right )_{+}  - \left ( \mathcal{A}_{1}\cdot x \right )_{-}=\mathcal{A}_{1}\cdot x
        ~and~
            \mathcal{A}_{1}\cdot \left(x\right)_{+} - \mathcal{A}_{1} \cdot\left ( x \right )_{-}=\mathcal{A}_{1}\cdot \left(x\right)_{+}=\mathcal{A}_{1}\cdot x; \]
        \item [(2)] if $x\in R_{-}^{n}/\left\{\theta \right\}$, then
        \[\left ( \mathcal{A}_{1}\cdot x \right )_{+}  - \left ( \mathcal{A}_{1}\cdot x \right )_{-}=\mathcal{A}_{1}\cdot x
        ~and~
                \mathcal{A}_{1}\cdot \left(x\right)_{+} - \mathcal{A}_{1} \cdot\left ( x \right )_{-}=- \mathcal{A}_{1} \cdot\left ( x \right )_{-}=\mathcal{A}_{1}\cdot x. \]
    \end{enumerate}
     Based on the above analysis, \eqref{eq:7} holds.

By \eqref{eq:6}, \eqref{eq:7} and \eqref{eq:8}, we obtain
\begin{equation*}
\left ( \mathcal{A}_{1} \cdot\left ( x \right )_{+}+\left (
\mathcal{A}_{1}\cdot x \right )_{+}-\mathcal{A}_{1} \cdot\left ( x
\right )_{+}  \right )\wedge \left ( \mathcal{A}_{2} \cdot\left ( x
\right )_{+}+\left ( \mathcal{A}_{2}\cdot x \right
)_{+}-\mathcal{A}_{2} \cdot\left ( x \right )_{+}  \right )=\theta
\end{equation*}
and
\begin{equation*}
\left ( \mathcal{A}_{1} \cdot\left ( x \right )_{-}+\left (
\mathcal{A}_{1}\cdot x \right )_{-}-\mathcal{A}_{1} \cdot\left ( x
\right )_{-}  \right )\wedge \left ( \mathcal{A}_{2} \cdot\left ( x
\right )_{-}+\left ( \mathcal{A}_{2}\cdot x \right
)_{-}-\mathcal{A}_{2} \cdot\left ( x \right )_{-}  \right )=\theta,
\end{equation*}
which implies that $\left ( x \right )_{+}$ and $\left ( x \right
)_{-}$ are two different solutions of the VTCP$(\tilde{\mathcal{A} },
\tilde{q})$. This contradicts that the VTCP$(\tilde{\mathcal{A} },
\tilde{q})$ has at most one solution. So, $\tilde{\mathcal{A} }$ is
of type $VP$-$\uppercase\expandafter{\romannumeral1}$.
    \end{proof}

\begin{lemma}\label{th:4.8}
If $\tilde{\mathcal{A} }=\left \{ \mathcal{A}_{1}, \mathcal{A}_{2}
\right \}$ is of type strong $VP$, then the VTCP$(\tilde{\mathcal{A}
}, \tilde{q})$ has at most one
solution for all $q_{1}, q_{2} \in R^{n}$.
\end{lemma}
\begin{proof}
If $\tilde{\mathcal{A} }$ is of type strong $VP$, then
$$\underset{i\in \left [ n \right ] }{\text{max}} \left (
\mathcal{A}_{1} x^{m-1}-\mathcal{A}_{1} y^{m-1}  \right )_{i} \left
( \mathcal{A}_{2} x^{m-1}-\mathcal{A}_{2} y^{m-1}\right )_{i}>0$$
holds for all $x,y\in R^{n}$ and $x\ne y$. Let $x$, $y$ be the two
different solutions of the VTCP$(\tilde{\mathcal{A} }, \tilde{q})$.
Similar to the proof of Theorem \ref{the: un}, \eqref{eq:3.7} holds.
This contradicts that $\tilde{\mathcal{A} }$ is of type strong $VP$.
So, the VTCP$(\tilde{\mathcal{A} }, \tilde{q})$ has at most one solution for all $q_{1},
q_{2} \in R^{n}$.
\end{proof}

Based on Proposition \ref{pro:1}, Proposition \ref{pro:3.3}, Theorem \ref{the: ex} and Lemma \ref{th:4.8}, we can obtain a sufficient condition for the existence and uniqueness of the solution of the VTCP$(\tilde{\mathcal{A} }, \tilde{q})$, see Theorem \ref{the:st}.
\begin{theorem} \label{the:st}
If $\tilde{\mathcal{A} }=\left \{ \mathcal{A}_{1}, \mathcal{A}_{2}
\right \}$ is of type strong $VP$ with $\mathcal{A}_{1},
\mathcal{A}_{2}  \in R^{\left [ m, n  \right ] }$, $m>2$ and
$\mathcal{A}_{2}$ is a $R$-tensor, then the VTCP$(\tilde{\mathcal{A}
}, \tilde{q})$  has a unique solution for all $q_{1}, q_{2} \in
R^{n}$.
\end{theorem}
Furthermore, based on Proposition \ref{pro:3.3} and Theorem
\ref{the:st}, Corollary 4.1 can be obtained.
\begin{corollary}\label{cro:vp2}
If $\tilde{\mathcal{A} }=\left \{ \mathcal{A}_{1}, \mathcal{A}_{2}
\right \}$ is of type $VP$-$\uppercase\expandafter{\romannumeral2}$
with $\mathcal{A}_{1}, \mathcal{A}_{2}  \in R^{\left [ m, n  \right
] }$, $m$ is an even number and greater than 2, and
$\mathcal{A}_{2}$ is a $R$ tensor, then the VTCP$(\tilde{\mathcal{A}
}, \tilde{q})$  has a unique solution for all $q_{1}, q_{2} \in
R^{n}$.
\end{corollary}

In fact, Corollary {\rm \ref{cro:vp2}} can also be obtained by
combining Proposition {\rm \ref{pro:1}}, Theorem {\rm \ref{the: ex}}
and Theorem {\rm \ref{the: un}}.

Now, we give the following example to illustrate Theorem \ref{the:st}, see Example \ref{EX:4.1}.
\begin{example}\label{EX:4.1}
    Let $\tilde{\mathcal{A} } = \left \{ \mathcal{A}_{1}, \mathcal{A}_{2}   \right \}$, where
    \begin{gather*}
        \mathcal{A}_{1}\left ( 1, 1, :, : \right )  =\begin{pmatrix}
            1  & 1\\-2
              &1
            \end{pmatrix},~\mathcal{A}_{1}\left (1, 2, :, : \right )=\begin{pmatrix}
             1 & -1\\0
              &0
            \end{pmatrix},
        \\\mathcal{A}_{1}\left ( 2, 1, :, : \right )  =\begin{pmatrix}
            0  & 0\\1
              &0
            \end{pmatrix},~\mathcal{A}_{1}\left (2, 2, :, : \right )=\begin{pmatrix}
             0 & -1\\1
              &1
            \end{pmatrix},
        \\\mathcal{A}_{2}\left ( 1, 1, :, : \right )  =\begin{pmatrix}
            1  & 0\\0
              &0
            \end{pmatrix},~\mathcal{A}_{2}\left (1, 2, :, : \right )=\begin{pmatrix}
             0 & 0\\0
              &0
            \end{pmatrix},
        \\\mathcal{A}_{2}\left ( 2, 1, :, : \right )  =\begin{pmatrix}
            0  & -1\\1
              &0
            \end{pmatrix},~\mathcal{A}_{2}\left (2, 2, :, : \right )=\begin{pmatrix}
             0 & 0\\0
              &1
            \end{pmatrix}.
    \end{gather*}
    Clearly, from Example \ref{ex:sp}, $\tilde{\mathcal{A} }= \left \{ \mathcal{A}_{1}, \mathcal{A}_{2}   \right \}$ is of type strong $VP$,
    \begin{equation*}
        \mathcal{A}_{1}x^{3}= \begin{pmatrix}
            x_{1}^{3}    \\x_{1}^{2}x_{2} +x_{2}^{3}
            \end{pmatrix}~and~
                \mathcal{A}_{2}x^{3} = \begin{pmatrix}
                    x_{1}^{3}   \\x_{2}^{3}

                \end{pmatrix}.
            \end{equation*}
Suppose $\theta \le x\ne\theta$ and $t \ge 0$, we discuss  two cases: $x_{1}>0$ and $x_{1}=0$.
\begin{enumerate}
    \item [(1)] if $x_{1}>0$, then $x_{1}^{3}+t>0; $
    \item [(2)] if $x_{1}=0$, then $x_{1}^{3}+t \ge 0$, but $x_{2}> 0,x_{2}^{3}+t>0. $
\end{enumerate}
Thus, $\mathcal{A}_{2}$ is a R-tensor. By the above analysis, $\tilde{\mathcal{A} }$ satisfies the conditions of Theorem {\rm \ref{the:st}}.

Now, we prove that the VTCP$(\tilde{\mathcal{A} }, \tilde{q})$ has a
unique solution for all $q_{1}, q_{2} \in R^{n}$. Let $q_{1}=(
a_{1},a_{2} )^{T}$ and $q_{2}=( b_{1},b_{2} )^{T}$. We only need to
prove that the following system has a unique solution
  \begin{equation*}
    \begin{cases}
        {{\rm min}} \left\{x_{1}^3+a_{1}, x_{1}^3+b_{1}\right\}=0,  \\ {{\rm min}} \left\{x_{1}^{2}x_{2}+x_{2}^3+a_{2}, x_{2}^3+b_{2}\right\}=0.
       \end{cases}
\end{equation*}
Note that
\begin{equation*}
{{\rm min}} \left\{x_{1}^3+a_{1}, x_{1}^3+b_{1}\right\}=0 \Longleftrightarrow  x_{1}^3+c=0,
\end{equation*}
where $c={{\rm min}} \left\{a_{1}, b_{1}\right\}$. So, $x_1=\sqrt[3]{-c} $.
Now, we show that
\begin{equation}\label{1}
    {{\rm min}} \left\{tx_{2}+x_{2}^3+a_{2}, x_{2}^3+b_{2}\right\}=0
    \end{equation}
    has a unique solution, where $t= \left(\sqrt[3]{-c}\right)^2 \ge 0$.
Let
\begin{equation*}
    f\left(x_{2}\right)=tx_{2}+x_{2}^3+a_{2}, g\left(x\right)=x_{2}^3+b_{2}.
\end{equation*}
Obviously, $f'(x_{2})\geq g'(x_{2})\geq 0$ for any $x_{2}\in R$. So
there is only one $x_{*}$ such that $f(x_{*})=g(x_{*})$. Clearly,
$f(x_{2})$ and $g(x_{2})$ both have only one zero point, denoted as
$x_{1}^{*}$ and $x_{2}^{*}$, respectively. We discusse two cases:
$f(x_{*}) = g(x_{*})\geq0$ and $f(x_{*}) = g(x_{*})<0$.
\begin{enumerate}
    \item [(1)] if $f(x_{*}) = g(x_{*})\geq0$, then the unique solution of \eqref{1} is $x_{2}=x_{1}^{*}$;
    \item [(2)] if $f(x_{*}) = g(x_{*})< 0$, then the unique solution of \eqref{1} is $x_{2}=x_{2}^{*}$.
\end{enumerate}
Thus, the VTCP$(\tilde{\mathcal{A} }, \tilde{q})$ has a unique
solution $x=\left(\sqrt[3]{-c}, x_{1}^{*} \right)^{T}$ or
$x=\left(\sqrt[3]{-c}, x_{2}^{*} \right)^{T}$.
\end{example}
Next, we turn to use the equal form of minimum function to obtain
the sufficient condition for the existence and uniqueness of
solution of the VTCP$(\tilde{\mathcal{A} }, \tilde{q})$. To achieve
this goal, the following lemmas are introduced.
\begin{lemma}\label{le:3.1}
If $\mathcal{A}_{1}$, $\mathcal{A}_{2} \in R^{\left [ m, n  \right ]
}$ are two $\mathcal{Z}$-tensors, then $D_{1} \cdot
\mathcal{A}_{1}+D_{2} \cdot \mathcal{A}_{2}$ is a
$\mathcal{Z}$-tensor, where $D_{1} \in R^{\left[2,n\right]}$ and
$D_{2}\in R^{\left[2,n\right]}$ are any diagonal matrices with the
principal diagonal elements greater than or equal to zero.
\end{lemma}
\begin{proof}
Let $D_{1}=\left(d_{i_{1}i_{2}}\right)$ be a diagonal matrix that
principal diagonal elements are $d_{1}^{\left(1\right)},
d_{2}^{\left(1\right)}, \cdots ,$ $d_{n}^{\left(1\right)}$ and
$d_{i}^{\left(1\right)}\ge 0$ for all $i\in \left[n\right]$. Then
\[\left(D_{1} \cdot \mathcal{A}_{1}\right)_{i\alpha_{1}}=\sum_{i_{2}=1}^{n}d_{ii_{2}}a^{\left ( 1 \right ) }_{i_{2}\alpha _{1} }=d_{i}^{\left(1\right)}a^{\left ( 1 \right ) }_{ i\alpha _{1}}, \]
where $i \in [n]$ and $\alpha _{1}\in \left [ n \right ]^{m-1}$.
Obviously, $D_{1} \cdot \mathcal{A}_{1}$ is a $\mathcal{Z}$-tensor.
Similarly, $D_{2} \cdot \mathcal{A}_{2}$ is also a
$\mathcal{Z}$-tensor. Hence, $D_{1} \cdot \mathcal{A}_{1}+D_{2}
\cdot \mathcal{A}_{2}$ is a $\mathcal{Z}$-tensor.
\end{proof}

\begin{lemma}\label{le:3.2}
If $\tilde{\mathcal{A} }=\left \{ \mathcal{A}_{1}, \mathcal{A}_{2}
\right \}$ is of type semi-positive, then $D_{1} \cdot
\mathcal{A}_{1}+D_{2} \cdot \mathcal{A}_{2}$ is a semi-positive
tensor, where $D_{1}$, $D_{2}\in R^{\left[2,n\right]}$ are any
diagonal matrices with the principal diagonal elements greater than
or equal to zero and  $D_{1}+D_{2} \ne \theta$.
\end{lemma}
\begin{proof}
Since $\tilde{\mathcal{A} }=\left \{ \mathcal{A}_{1},
\mathcal{A}_{2}   \right \}$ is of type semi-positive, there exists
a vector $ x>\theta$ such that $\mathcal{A}_{1} x^{m-1}>\theta$ and
$\mathcal{A}_{2} x^{m-1}>\theta$. Obviously,
\begin{equation*}
\left(D_{1} \cdot\mathcal{A}_{1}+D_{2} \cdot \mathcal{A}_{2}\right) \cdot x
=D_{1} \cdot\mathcal{A}_{1}\cdot x+D_{2} \cdot \mathcal{A}_{2}\cdot x
=D_{1} \cdot\left(\mathcal{A}_{1} x^{m-1}\right)+
D_{2} \cdot \left(\mathcal{A}_{2} x^{m-1}\right) >\theta,
\end{equation*}
where $D_{1}$, $D_{2}\in R^{\left[2,n\right]}$ are any diagonal
matrices with the principal diagonal elements greater than or equal
to zero and  $D_{1}+D_{2} \ne \theta$.
\end{proof}

\begin{lemma}[\rm \cite{ding2003mtensor}]\label{le:3.3}
A $\mathcal{Z}$-tensor is a strong $\mathcal{M}$-tensor if snd only
if it is semi-positive.
\end{lemma}

\begin{lemma}[\cite{wu2022perturbation}] \label{le:1}
Let all $a_{i}$, $b_{i}\in R$, $i=1, 2, \cdots, n$. Then there exist
$\lambda_{i}\in\left[0, 1\right]$ with  ${\textstyle
\sum_{i=1}^{n}}\lambda _{i}=1$ such that
\begin{equation*}
\underset{i\in \left [ n \right ] }{{\rm min}}
\left\{a_{i}\right\}-\underset{i\in \left [ n \right ] }{{\rm
min}\left\{b_{i}\right\}}=\sum_{i=1}^{n}
\lambda_{i}\left(a_{i}-b_{i}\right).
\end{equation*}
\end{lemma}
From Lemma \ref{le:1}, Lemma \ref{the: min} can be easily derived.
\begin{lemma}\label{the: min}
Let $\mathcal{A}_{1}$, $\mathcal{A}_{2} \in T^{m, n}$ and $q_{1}$,
$q_{2}\in R^{n}$. Then there exist $D_{1}$ and $D_{2}$ such that
\begin{equation*}
{{\rm min}} \left\{q_{1}+\mathcal{A}_{1} x^{m-1},
q_{2}+\mathcal{A}_{2} x^{m-1}\right\}-{{\rm min}\left\{0,
0\right\}}=D_{1}\cdot\left(q_{1}+\mathcal{A}_{1}
x^{m-1}\right)+D_{2}\cdot\left(q_{2}+\mathcal{A}_{2} x^{m-1}\right),
\end{equation*}
where $D_{1}$, $D_{2}$ are diagonal matrices that diagonal elements
are $d^{\left(1\right)}=\left(d_{1}^{\left(1\right)},
d_{2}^{\left(1\right)}, \cdots , d_{n}^{\left(1\right)}\right)$
$\in$ $\left[0, 1\right]^{n}$ and
$d^{\left(2\right)}=\left(d_{1}^{\left(2\right)},
d_{2}^{\left(2\right)}, \cdots , d_{n}^{\left(2\right)}\right) \in
\left[0, 1\right]^{n}$, respectively, and $D_{1}+D_{2}=I$.
\end{lemma}
\begin{theorem}\label{the:4.8}
Let $\tilde{\mathcal{A} }=\left \{ \mathcal{A}_{1}, \mathcal{A}_{2}
\right \}$ be of type semi-positive, where  $\mathcal{A}_{1}$,
$\mathcal{A}_{2} \in R^{\left [ m, n  \right ] }$ are two
$\mathcal{Z}$-tensors. Then the VTCP$(\tilde{\mathcal{A} },
\tilde{q})$ has a unique  positive solution for any $q_{1}$, $q_{2}
< \theta$.
\end{theorem}
\begin{proof}
By Lemma \ref{the: min}, the VTCP$(\tilde{\mathcal{A} }, \tilde{q})$
is equivalent to find a vector $x$ such that
\begin{equation*}
\left(D_{1} \cdot \mathcal{A}_{1}+D_{2} \cdot
\mathcal{A}_{2}\right)x^{m-1}=-D_{1}q_{1}-D_{2}q_{2},
\end{equation*}
where  $D_{i}$ is a diagonal matrix and the principal diagonal
elements $d^{\left(i\right)}:=(d_{1}^{\left(i\right)},
d_{2}^{\left(i\right)}, \cdots, d_{n}^{\left(i\right)})$ $\in
\left[0, 1\right]^{n}$ with $i \in \left\{1, 2\right\}$ and
$D_{1}+D_{2}=I$. By Lemmas \ref{le:3.1}-\ref{le:3.3}, $D_{1} \cdot
\mathcal{A}_{1}+D_{2} \cdot \mathcal{A}_{2}$ is a strong
$\mathcal{M}$-tensor. Note that $-D_{1}q_{1}-D_{2}q_{2}>\theta$.
Thus, by Theorem 3.2 in \cite{ding2016solving},
\begin{equation*}
\left(D_{1} \cdot \mathcal{A}_{1}+D_{2} \cdot
\mathcal{A}_{2}\right)x^{m-1}=-D_{1}q_{1}-D_{2}q_{2}
\end{equation*}
has a unique positive solution, which is equivalent to that the
VTCP$(\tilde{\mathcal{A} }, \tilde{q})$  has a unique positive
solution.
\end{proof}

Now, we give the following example to illustrate Theorem \ref{the:4.8}, see Example \ref{ex:4.2}.
\begin{example}\label{ex:4.2}
 Let $\mathcal{A}_{1}$, $\mathcal{A}_{2} \in R^{\left[3, 2\right]}$, where
 \begin{gather*}
            \mathcal{A}_{1}\left ( 1, :, : \right )  =\begin{pmatrix}
                1 & 0\\0
                  &0
                \end{pmatrix},~\mathcal{A}_{1}\left ( 2, :, : \right )=\begin{pmatrix}
                 -1 & 0\\0
                  &1
                \end{pmatrix},
            \\\mathcal{A}_{2}\left ( 1, :, : \right )  =\begin{pmatrix}
                    1 & 0\\0
                      &0
                    \end{pmatrix},~\mathcal{A}_{2}\left ( 2, :, : \right )=\begin{pmatrix}
                    -1 & -1\\0
                      &1
                \end{pmatrix}.
            \end{gather*}
Obviously, $\mathcal{A}_{1}$ and $\mathcal{A}_{2}$ are two
$\mathcal{Z}$-tensors. By the calculation, we obtain
            \begin{equation*}
                \mathcal{A}_{1}x^{2} = \begin{pmatrix}
                    x_{1}^{2}   \\-x_{1}^{2} +x_{2}^{2}

                    \end{pmatrix}~and~
                \mathcal{A}_{2}x^{2} = \begin{pmatrix}
                    x_{1}^{2}   \\-x_{1}^{2}-x_{1}x_{2} +x_{2}^{2}

                    \end{pmatrix}.
\end{equation*}
Let $x=(1,2)^{T}$. Then $\mathcal{A}_{1}x^{2}=(1,3)^{T}$ and
$\mathcal{A}_{2}x^{2}=(1,1)^{T}$. Thus, $\tilde{\mathcal{A} }=\left
\{ \mathcal{A}_{1}, \mathcal{A}_{2}   \right \}$ is of type
semi-positive. By the above analysis, $\tilde{\mathcal{A} }$
satisfies the conditions of Theorem {\rm \ref{the:4.8}}.

Now, we prove that the VTCP$(\tilde{\mathcal{A} }, \tilde{q})$  has
a unique positive solution for any $q_{1}$, $q_{2} < \theta$. Let
$q_{1}=(a_{1},a_{2})^{T} < \theta $ and $q_{2}=(b_{1},b_{2} )^{T} <
\theta$. We only need to prove that the following system has a
unique positive solution
\begin{equation*}
\begin{cases}
{{\rm min}} \left\{x_{1}^2+a_{1}, x_{1}^2+b_{1}\right\}=0,  \\ {{\rm
min}} \left\{-x_{1}^{2} +x_{2}^{2}+a_{2}, -x_{1}^{2}-x_{1}x_{2}
+x_{2}^{2}+b_{2}\right\}=0.
\end{cases}
\end{equation*}
Note that
\begin{equation}\label{2}
{{\rm min}} \left\{x_{1}^2+a_{1}, x_{1}^2+b_{1}\right\}=0
\Longleftrightarrow  x_{1}^2+c=0,
\end{equation}
where $c={{\rm min}} \left\{a_{1}, b_{1}\right\}<0$. So, the unique
positive solution of \eqref{2} is $x_1=\sqrt{-c} $. Now, we prove that
\begin{equation}\label{eq:4.11}
    {{\rm min}} \left\{c +a_{2}+x_{2}^{2}, c +b_{2}-\sqrt{-c} x_{2}+x_{2}^{2}\right\}=0
    \end{equation}
    has a unique positive solution. Let
    \begin{equation*}
        f\left(x_{2}\right)=c +a_{2}+x_{2}^{2}~and~ g\left(x_{2}\right)=c +b_{2}-\sqrt{-c} x_{2}+x_{2}^{2}.
    \end{equation*}
Since $f(0)=c +a_{2}<0$ and $g(0)=c +b_{2}<0$, both $f(x_{2})$ and
$g(x_{2})$ have only two zero points with different signs. It is
clear that \eqref{eq:4.11} has a unique positive solution
$x_{2}=x_{*}$. Thus, the VTCP$(\tilde{\mathcal{A} }, \tilde{q})$ has
a unique positive solution $x=\left(\sqrt{-c}, x_{*} \right)^{T}$.
\end{example}
More sufficient conditions for the existence of a unique solution of the tensor equation $\left(D_{1} \cdot \mathcal{A}_{1}+D_{2} \cdot \mathcal{A}_{2}\right)x^{m-1}=-D_{1}q_{1}-D_{2}q_{2}$ can be found in \cite{liu2018tensors}.

Finally, by the way, we present a sufficient condition for the
existence and uniqueness of the solution of the
VTCP$(\tilde{\mathcal{A} }, \tilde{q})$ when $\tilde{\mathcal{A}
}=\left \{ \mathcal{A}_{1}, \mathcal{A}_{2}   \right \}$ is of type
$VP$-$\uppercase\expandafter{\romannumeral2}$. For this, Lemma 4.7
is required.
\begin{lemma} \label{le: ex2}
Assume  that $\tilde{\mathcal{A} }= \left \{ \mathcal{A}_{1},
\mathcal{A}_{2}   \right \}$ is of type $VR_{0}$, where
$\mathcal{A}_{1}, \mathcal{A}_{2}  \in R^{\left [ m, n  \right ] }$
and $m$ is an even number. If $\mathcal{A}_{1} $ has an even order
right inverse $\mathcal{A}_{1}^{R_{k}} \in R^{\left [ k, n  \right ]
}$ such that $\mathcal{A}_{2}\cdot \mathcal{A}_{1}^{R_{k}}$ is a
$R$-tensor, then $\tilde{\mathcal{A} }$ is of type $VQ$. \end{lemma}
\begin{proof}
Let
\begin{equation*}
G\left (x, t \right )=x\wedge\left(\mathcal{A}_{2} \cdot
\mathcal{A}_{1}^{R_{k}} \cdot x ^{\left [ \frac{1}{\left ( m-1
\right )\left ( k-1 \right )  }  \right ]}+td\right)
\end{equation*}
and
 \begin{equation*}
            H\left ( x, t\right)=x\wedge\left(\mathcal{A}_{2} \cdot \mathcal{A}_{1}^{R_{k}} \cdot \left ( x-tq_{1} \right )  ^{\left [ \frac{1}{\left ( m-1 \right )\left ( k-1 \right )  }  \right ]}+tq_{2}\right),
 \end{equation*}
 where $t \in [0, 1]$, $d> 0$, $q_{1},q_{2}\in R^{n}$.
        Obviously, for a given constant $t^{*}\in [0, 1]$, if vector $y$ is a solution of $G\left (x, t^{*} \right )=\theta$,  then there is a solution
        $x=\mathcal{A}_{1}^{R_{k}}\cdot y ^{\left [ \frac{1}{\left ( m-1 \right )\left ( k-1 \right )  }  \right ]}$
        of $\mathcal{A}_{1} x^{m-1}\wedge \left(t^{*}d+\mathcal{A}_{2} x^{m-1}\right)=\theta$ corresponding to it. If $y$ is a solution of $H\left ( x, t^{*}\right)$, then $x=\mathcal{A}_{1}^{R_{k}}\cdot \left(y-t^{*}q_{1}\right) ^{\left [ \frac{1}{\left ( m-1 \right )\left ( k-1 \right )  }  \right ]} $ is a solution of
        \begin{equation*}
            \left(t^{*}q_{1}+\mathcal{A}_{1} x^{m-1}\right)\wedge \left(t^{*}q_{2}+\mathcal{A}_{2} x^{m-1}\right)=\theta.
        \end{equation*}
        By Theorem \ref{thm: bound}, the solution sets of $G\left (x, t \right )=\theta$ and $H\left (x, t \right )=\theta$ are bounded for any $t \in [0, 1]$.
        So, there exists a bounded open set $\Omega$ that contains the sets
        \begin{equation*}
            X_{1}=\left\{x: G\left ( x, t\right)=\theta, t \in [0, 1]\right\}~and~X_{2}=\left\{x: H\left ( x, t\right)=\theta, t \in [0, 1]\right\} ,
        \end{equation*}
        and $\left ( X_{1}\cup X_{2} \right ) \cap \partial\Omega =\emptyset $.
        Since $\mathcal{A}_{2}\cdot \mathcal{A}_{1}^{R_{k}}$ is a $R$-tensor, there is no solution to the following system
        \begin{equation}\label{3}
            \begin{cases}
                \theta \ne x \ge 0, t\ge0, \\\left ( \mathcal{A}_{2}\cdot \mathcal{A}_{1}^{R_{k}}\cdot x \right )_{i} +t=\theta , if~x_{i}>\theta ,
                \\\left ( \mathcal{A}_{2}\cdot \mathcal{A}_{1}^{R_{k}}\cdot x \right )_{j} +t\ge \theta , if~x_{j}=\theta,
               \end{cases}
        \end{equation}
        which implies that there is no solution to the following system
        \begin{gather}
            \begin{cases}
                \theta \ne x \ge 0, t\ge0, \\\left ( \mathcal{A}_{2}\cdot \mathcal{A}_{1}^{R_{k}}\cdot x ^{\left [ \frac{1}{\left ( m-1 \right )\left ( k-1 \right )  }  \right ]}\right )_{i} +t=\theta , if~x_{i}>\theta ,
                \\\left ( \mathcal{A}_{2}\cdot \mathcal{A}_{1}^{R_{k}}\cdot x^{\left [ \frac{1}{\left ( m-1 \right )\left ( k-1 \right )  }  \right ]} \right )_{j} +t\ge \theta , if~x_{j}=\theta.
               \end{cases} \label{eq:9}
        \end{gather}
        In fact, assume that there is an $x$ that satisfies \eqref{eq:9}.
        Let $y=x^{\left [ \frac{1}{\left ( m-1 \right )\left ( k-1 \right )  }  \right ]}$. Then
        \begin{equation*}
            x_{i}>\theta \Leftrightarrow  y_{i}>\theta, x_{i}=\theta \Leftrightarrow  y_{i}=\theta.
        \end{equation*}
        Thus, $y$ is a solution of \eqref{3}, which contradicts that $\mathcal{A}_{2}\cdot \mathcal{A}_{1}^{R_{k}}$ is a $R$-tensor.
        From \eqref{eq:9},
        \begin{equation*}
            G\left (x, t \right )=x\wedge\left(\mathcal{A}_{2} \cdot \mathcal{A}_{1}^{R_{k}} \cdot x ^{\left [ \frac{1}{\left ( m-1 \right )\left ( k-1 \right )  }  \right ]}+td\right)=\theta
        \end{equation*}
        has only the zero solution for any $t \in [0, 1]$ and $d> 0$.
        Since $d>0$,
        \begin{equation*}
            \mathcal{A}_{2} \cdot \mathcal{A}_{1}^{R_{k}} \cdot x ^{\left [ \frac{1}{\left ( m-1 \right )\left ( k-1 \right )  }  \right ]}+d>\theta~and~G\left ( x, 1  \right )=x\wedge\left(\mathcal{A}_{2} \cdot \mathcal{A}_{1}^{R_{k}} \cdot x ^{\left [ \frac{1}{\left ( m-1 \right )\left ( k-1 \right )  }  \right ]}+d\right)=x
        \end{equation*}
        for $x$ near $\theta$.
        Thus, $degree\left(G\left ( x, 1  \right ), \Omega, \theta\right) =1$.
        By the homotopy invariance property of the degree \cite{1978Degree},
        \begin{equation*}
            degree\left(G\left ( x, 1  \right ), \Omega, \theta\right) =degree\left(G\left ( x \right ), \Omega, \theta\right)=1,
        \end{equation*}
        where $G\left ( x \right )=G\left ( x, 0 \right )$.
        Obviously, $H\left ( x \right )=G\left ( x \right )$, where $H\left ( x \right )=H\left ( x, 0 \right )$. By the homotopy invariance theorem \cite{1978Degree},
        \[degree\left(H\left ( x, 1 \right ), \Omega, \theta\right)=degree\left(H\left ( x \right ), \Omega, \theta\right)=degree\left(G\left ( x  \right ), \Omega, \theta\right)=1. \]
        Thus, $\tilde{\mathcal{A} }$ is of type $VQ$.
 \end{proof}

Combining Theorem {\rm \ref{the: un}} and Lemma {\rm \ref{le: ex2}},
Proposition 4.1 holds.

\begin{proposition}
Assume that $\tilde{\mathcal{A} }=\left \{ \mathcal{A}_{1},
\mathcal{A}_{2}   \right \}$ is of type
$VP$-$\uppercase\expandafter{\romannumeral2}$ with $\mathcal{A}_{1},
\mathcal{A}_{2}  \in R^{\left [ m, n  \right ] }$ and $m$ being an
even number. If $\mathcal{A}_{1} $ has an even order right inverse
$\mathcal{A}_{1}^{R_{k}} \in R^{\left [ k, n  \right ] }$ such that
 $\mathcal{A}_{2}\cdot \mathcal{A}_{1}^{R_{k}}$ is a $R$-tensor, then the VTCP$(\tilde{\mathcal{A} }, \tilde{q})$ has a unique solution for all $q_{1}, q_{2} \in
 R^{n}$.
 \end{proposition}

\section{Conclusions}
In this paper, we have introduced the vertical tensor
complementarity problem (VTCP), and defined some tensor sets with
special structures. We obtained an equivalence condition, i.e., the
tenor set $\tilde{\mathcal{A} }$ is of type $VR_{0}$ if and only if
the solution set of the VTCP$(\tilde{\mathcal{A} }, \tilde{q})$ is
bounded. Meanwhile, we obtained some sufficient conditions for the
existence and uniqueness of the solution of the VTCP from two
aspects: the degree theory and the equal form of minimum function.

%\section*{Acknowledgements}
%The authors would like to thank Editor Prof. Yang Zhihua for his
%helpful suggestions.

 \end{document}